\documentclass{article}
\usepackage{amsmath}
\usepackage{amsthm}
\usepackage{amstext}
\usepackage{amssymb}
\usepackage{dsfont}
\usepackage{graphics}
\usepackage{graphicx}
\usepackage{color}
\usepackage{pdfrender}
\usepackage{mathtools}
\usepackage{tikz}
\usepackage[title]{appendix}
\usepackage[nottoc,numbib]{tocbibind}

\newtheorem{theorem}{Theorem}[section]

\newtheorem{lemma}[theorem]{Lemma}

\newtheorem{innercustomthm}{Theorem}
\newenvironment{customthm}[1]
  {\renewcommand\theinnercustomthm{#1}\innercustomthm}
  {\endinnercustomthm}
\newcommand{\keywords}[1]{\textbf{Keywords:}\quad #1}

\def\multiset#1#2{\ensuremath{\left(\kern-.3em\left(\kern-.1em\genfrac{}{}{0pt}{}{#1}{#2}\kern-.1em\right)\kern-.3em\right)}}
\def\dispmultiset#1#2{\ensuremath{\left(\kern-.5em\left(\kern-.25em\genfrac{}{}{0pt}{}{#1}{#2}\kern-.25em\right)\kern-.5em\right)}}

\newcommand{\cupdot}{\mathbin{\mathaccent\cdot\cup}}

\begin{document}
\pdfrender{StrokeColor=black,TextRenderingMode=2,LineWidth=0.2pt}
\begin{center}
\Large The number of non-isomorphic arithmetic expressions that can be constructed using $+$, $-$, $\times$ and $\div$\\[0.2cm]
\normalsize Boaz Cohen\\
\verb|arctanx@gmail.com|\\[0.2cm]
Department of Computer Science,\\
The Academic College of Tel-Aviv, \\
Rabenu Yeruham St., P.O.B 8401 Yaffo, 6818211, Israel\\
\end{center}
\begin{abstract}
The goal of this paper is to count the number of distinct functions of $n$ variables, up to permutation of the variables, that can be constructed using each variable exactly once, without constants, using only the operations of addition, subtraction, multiplication, and division. We refer to such a function as an \textit{arithmetic expression}. Under this definition, two expressions are \textit{identical} if they represent the same rational function; for example, $x_1-x_2-x_3$ and $x_1-(x_2+x_3)$ are identical arithmetic expressions, as are $x_1(x_2+x_3)$ and $(x_2+x_3)x_1$. Two arithmetic expressions are said to be \textit{isomorphic} if one can be obtained from the other by a permutation of the variables. For example, $(x_1-x_2)/x_3$ and $(x_2-x_3)/x_1$ are isomorphic. The first few values of the number of non-isomorphic arithmetic expressions with $n$ variables are:
\[
1,4,18,93,500,2844,16621,99674,608448,\dots
\]
In order to accomplish this enumeration, we classify the set of all arithmetic expressions into 12 disjoint categories. Counting all non-isomorphic expressions in each category allows us to obtain the total required quantity.
\end{abstract}
\keywords{arithmetic expression, projectively extended real line, Generalized multivariate functions, multilinear polynomials}
%
\section{Introduction}
A well-known family of puzzles involves generating a specific target value using a given list of numbers and basic arithmetic operators. A challenging member of this family is the ``21-puzzle":
\begin{quote}
\textit{Obtain the number $21$ using the numbers $1,5,6,7$, where each of the numbers may be used only once together with the four basic binary arithmetic operations $+,-,\times ,\div$ and parentheses.}
\end{quote}
The answer to this puzzle is given in Section~5 of this study. A brute-force approach to solving this kind of puzzle would be to iterate through all possible \textit{arithmetic expressions} (abbreviated as AE's). These are, loosely speaking, all functions consisting of $n$ variables defined by means of the basic binary arithmetic operators $+$, $-$, $\times$, and $\div$ together with parentheses, where each of the $n$ variables may be used only once. Under this definition, we distinguish between the syntactic form of an expression and the semantic function it represents; two expressions are considered \textit{identical} if they represent the same rational function. For example, $x_1-x_2-x_3$ and $x_1-(x_2+x_3)$ are identical AE's, as are $x_1(x_2+x_3)$ and $(x_2+x_3)x_1$. Crucial to this definition is that these four operations are used strictly as binary operators. We do not permit unary operations, such as negation ($-\blacksquare$) or the multiplicative inverse ($\blacksquare^{-1}$), unless the resulting function \textit{can} be represented using only binary operators. For instance, the function $-(x_1x_2^{-1}-x_3x_4)$ qualifies as an AE because it can be rewritten as $x_3x_4-x_1/x_2$, which does not use negation or the multiplicative inverse operator. Examples of such AE's with four variables include:
\[
(x_1+x_2)(x_3-x_4)\qquad x_1+\frac{x_2x_3}{x_4}\qquad \frac{x_1}{x_2+x_3+x_4}\qquad\text{etc.}
\]
In contrast, it can be shown that $(x_1-x_2)^{-1}$ and $-x_1-x_2x_3$ are not AE's, as they cannot be expressed solely through the four permitted binary operations.

In this approach, to decrease search time, it is reasonable to ignore the variables themselves and focus only on the form of the function. For example, the functions $f(x_1,x_2,x_3)=x_1+x_2-x_3$ and $g(x_1,x_2,x_3)=x_1-x_2+x_3$ are,
of course, different, but they should not be distinguished as distinct forms since they describe, in a certain sense, the same operation: namely, subtracting one variable from the sum of the other two. Formally, two AE's are said to be \textit{isomorphic} if one can be obtained from the other by a permutation of the variables. For example, the following AE's are isomorphic:
\[
f(x_1,x_2,x_3,x_4)=\frac{x_1}{x_2-x_3}+x_4\qquad
g(x_1,x_2,x_3,x_4)=x_2-\frac{x_4}{x_3-x_1}
\]
Indeed, by choosing the permutation $\sigma=(x_1\ x_2\ x_4)$, we obtain
\[
\begin{split}
  g(\sigma(x_1),\sigma(x_2),\sigma(x_3),\sigma(x_4))
  &=g(x_2,x_4,x_3,x_1)\\
  &=x_4-\frac{x_1}{x_3-x_2}=\frac{x_1}{x_2-x_3}+x_4=f(x_1,x_2,x_3,x_4).
\end{split}
\]
The goal of this paper is to count the number of non-isomorphic AE's consisting of $n$ variables. For example, in the case of two variables, there are $4$ non-isomorphic AE's, namely
\[
x_1+x_2 \qquad x_1-x_2 \qquad x_1x_2 \qquad \dfrac{x_1}{x_2}
\]
In the case of three variables, there are $18$ non-isomorphic AE's, which are listed below:
\[
\begin{array}{cccccc}
 x_1+x_2+x_3 & x_1x_2x_3 & \dfrac{x_1+x_2}{x_3} & \dfrac{x_1}{x_2+x_3} & \dfrac{x_1}{x_2x_3}  & x_1x_2+x_3 \\[0.4cm]
 x_1+x_2-x_3 & (x_1+x_2)x_3 & \dfrac{x_1-x_2}{x_3} & \dfrac{x_1}{x_2}+x_3 & \dfrac{x_1x_2}{x_3}  & x_1x_2-x_3\\[0.4cm]
 x_1-x_2-x_3 & (x_1-x_2)x_3    & \dfrac{x_1}{x_2-x_3} & \dfrac{x_1}{x_2}-x_3 & x_1-\dfrac{x_2}{x_3} & x_1-x_2x_3\\[0.4cm]
\end{array}
\]
As we shall see in this paper, further values of this sequence are~\cite{Sloane_A393077}
\[
1,4,18,93,500,2844,16621,99674,608448,\ldots
\]
To the best of our knowledge, this problem has not been addressed so far. However, the problem of the number of AE's which are generated only with $+$ and $\times$ is extensively discussed in the literature and in a variety of different contexts. Riordan and Shannon~\cite{Riordan} and {\L}omnicki~\cite{Lomnicki} investigated this problem and, in particular, gave recurrences to compute the number of non-isomorphic AE's generated only with $+$ and $\times$. Every such AE corresponds to the so-called \textit{two-terminal series-parallel networks} with $n$ edges. These objects appear in the study of electrical networks built of switches and resistances, as well as in the investigation of so-called event networks. See~\cite{Duffin,Isokawa1,Isokawa2,MacMahon1,MacMahon2,Moon}. For example, there are $10$ such AE's with $4$ variables, namely
\[
\begin{array}{cc}
  x_1x_2x_3x_4 &  x_1+x_2+x_3+x_4\\
  x_1x_2(x_3+x_4) & x_1+x_2+x_3x_4 \\
  x_1(x_2+x_3+x_4) & x_1+x_2x_3x_4 \\
  x_1(x_2+x_3x_4) & x_1+x_2(x_3+x_4)\\
  (x_1+x_2)(x_3+x_4) & x_1x_2+x_3x_4 \\
\end{array}
\]
The first nine values of this sequence are $1,2,4,10,24,66,180,522,1532$. For more values, see~\cite{Sloane_A000084}.

This paper is structured as follows: In Section~2, we give the required notation and definitions needed in our analysis. In Section~3, we present the fundamental theorem of arithmetic expressions, which describes, in a certain sense, how AE's can be uniquely classified. This theorem is the main tool which allows us to perform the enumeration of non-isomorphic AE's. The rigorous proofs of the fundamental theorem have been put in Appendix A at the end of the paper. Our main result, in which we give explicit formulas for the number of non-isomorphic AE's, is presented in Section~4.
\section{Basic concepts and definitions}
In this section, we give the basic definitions and notation underlying our paper. Additionally, we describe the tools that we shall use in order to prove our main theorems.
\subsection{Preliminaries: the projectively extended real line}
In this article, it will be convenient to conduct our discussion in the framework of the \textit{projectively extended real line} $\widehat{\mathds{R}}\coloneqq\mathds{R}\cup\{\infty\}$, which is the extension of the field of real numbers $\mathds{R}$ by the point of infinity $\infty$. The connection of $\infty$ with the finite real numbers is established by setting $x+\infty=\infty+x=\infty$ for all finite $x$, and $x\cdot\infty=\infty\cdot x=\infty$ for all $x\neq0$, including $x=\infty$. We further define $x-\infty=\infty-x=\infty$ for all finite $x$. It is impossible, however, to define $\infty+\infty$, $\infty-\infty$, and $0\cdot\infty$ without violating the laws of arithmetic. Unlike most mathematical models of numbers, this structure allows division by zero, so $x/0=\infty$ for $x\neq0$ and $x/\infty=0$ for $x\neq\infty$. The expressions $0/0$ and $\infty/\infty$ are left undefined. Topologically, $\widehat{\mathds{R}}$ is homeomorphic to a circle. This can be visualized by a map that sends every point $a\neq\infty$ on the circle, along the line connecting it with $\infty$, to the equatorial line at the point $A$.
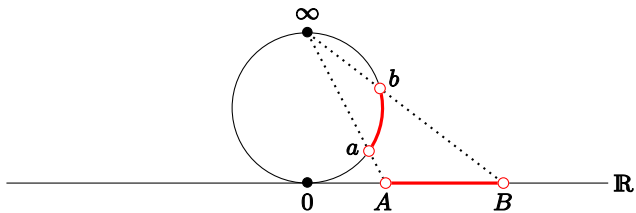
\begin{center}
\begin{tikzpicture}
        \draw (0,0) circle (1);
        \draw  (-4,-1)--(4,-1);
        \node at (0,1) {\small $\bullet$};
        \node at (0,-1) {\small $\bullet$};
        \node at (0,1.25) {\small $\infty$};
        \node at (1,-1.25) {\small $A$};
        \node at (0,-1.25) {\small $0$};
        \node at (2.6,-1.25) {\small $B$};
        \node at (0.6,-0.55) {\small $a$};
        \node at (1.15,0.4) {\small $b$};
        \node at (4.2,-1) {\small $\mathds{R}$};
        \draw[very thick, red] ({2*cos(-35)/(1-sin(-35))},-1)--({2*cos(15)/(1-sin(15))},-1);
        \draw[thick,dotted](0,1)--({cos(15)},{sin(15)})--({2*cos(15)/(1-sin(15))},-1);
        \draw[thick,dotted](0,1)--({cos(-35)},{sin(-35)})--({2*cos(-35)/(1-sin(-35))},-1);
        \draw[very thick, red] ({cos(15)},{sin(15)}) arc (15:-35:1);
        \filldraw[white] ({cos(15)},{sin(15)}) circle(2pt);\draw[red] ({cos(15)},{sin(15)}) circle(2pt);
        \filldraw[white] ({cos(-35)},{sin(-35)}) circle(2pt);\draw[red] ({cos(-35)},{sin(-35)}) circle(2pt);
        \filldraw[white] ({2*cos(15)/(1-sin(15))},-1) circle(2pt);\draw[red] ({2*cos(15)/(1-sin(15))},-1) circle(2pt);
        \filldraw[white] ({2*cos(-35)/(1-sin(-35))},-1) circle(2pt);\draw[red] ({2*cos(-35)/(1-sin(-35))},-1) circle(2pt); circle(2pt);
\end{tikzpicture}
\end{center}
The tools of calculus can be used to analyze functions over $\widehat{\mathds{R}}$. The definitions are motivated by the topology of this space. A \textit{neighbourhood} of a point $a\in\widehat{\mathds{R}}$ may be viewed as any open arc on the circle that contains $a$. If $f:D\to\widehat{\mathds{R}}$ is a function, where $D\subseteq\widehat{\mathds{R}}$, and if $L,a\in\widehat{\mathds{R}}$, then $\lim_{x\to a}f(x)=L$ if for every neighbourhood $\mathcal{N}$ of $L$ there exists a punctured neighbourhood $\mathcal{M}$ of $a$ such that $f(x)\in \mathcal{N}$ for every $x\in \mathcal{M}$. We say that $f$ is \textit{continuous at} $a\in D$ if $\lim_{x\to a}f(x)=f(a)$. For example, according to these definitions, the reciprocal function $f(x)=1/x$ is continuous at every $a\in\widehat{\mathds{R}}$. In fact, every rational function can be prolonged, in a unique way, to a function from $\widehat{\mathds{R}}$ to $\widehat{\mathds{R}}$ that is continuous on $\widehat{\mathds{R}}$. Indeed, since every rational function can be expressed as a finite composition $f_1\circ f_2\circ\cdots \circ f_n$, where each $f_i$ is either the reciprocal function or a polynomial, it suffices to verify that every  polynomial can be prolonged to a continuous function over all of $\widehat{\mathds{R}}$. This can be done by defining $f(\infty)=\infty$ for every non-constant polynomial. For example, $f(x)=x^2-x$ is a polynomial which is initially undefined at $\infty$. Nevertheless, since $\lim_{x\to\infty}(x^2-x)=\infty$, it can be uniquely prolonged to the function
\[
f(x)=
\begin{cases}
  x^2-x & x\neq\infty\\
  \infty & x=\infty
\end{cases}
\]
which is continuous on $\widehat{\mathds{R}}$. An important consequence of this discussion regards the continuity of the derivative of rational functions: since every derivative of a rational function is again rational, it follows that \textit{every rational function can be prolonged to a function from $\widehat{\mathds{R}}$ to $\widehat{\mathds{R}}$ that has a continuous derivative on $\widehat{\mathds{R}}$.}
\subsection{Generalized multivariate functions}
Let $\widehat{\mathds{R}}^{\infty}$ be the space of \textit{finite sequences over} $\widehat{\mathds{R}}$, that is,
\[
\widehat{\mathds{R}}^{\infty}\coloneqq
\big\{(a_1,a_2,a_3,\ldots):\hbox{$a_i\in\widehat{\mathds{R}}$ for all $i$ and $a_i\neq0$ for finitely many $i$'s}\big\}.
\]
A \textit{generalized multivariate function} is a function whose domain is a subset of $\widehat{\mathds{R}}^{\infty}$. For example, the function $f(x_1,x_2,\ldots)=x_3/x_1$ is a generalized multivariate function whose domain is
\[
\mathrm{dom}(f)=\big\{(a_1,a_2,\ldots)\in\widehat{\mathds{R}}^{\infty}:\hbox{$(a_1,a_3)\neq(0,0)$ and $(a_1,a_3)\neq(\infty,\infty)$}\big\}.
\]
To express generalized multivariate functions, we shall use the usual functional notation,
$f(x_1,x_2,x_3,\ldots)$ or $f(X)$, in short. However, to simplify the notation, especially when it is not necessary to emphasize the assigned values of the function, we shall just use its name - for example, $f=x_3/x_1$ instead of $f(x_1,x_2,\ldots)=x_3/x_1$. If $f$ and $g$ are generalized multivariate functions and $\ast\in\{+,-,\times,\div\}$, then $f\ast g$ is the function defined by $(f\ast g)(X)\coloneqq f(X)\ast g(X)$. Two important families of generalized multivariate functions are the \textit{constant} functions, $f(x_1,x_2,\ldots)=C$, where $C\in\widehat{\mathds{R}}$, and the \textit{atomic} functions, $f(x_1,x_2,\ldots)=x_i$, where $i$ is a positive integer.

We shall say that a generalized multivariate function $f$ \textit{depends upon} the variable $x_i$ if there exist two tuples $X_a=(c_1,\ldots,c_{i-1},a,c_{i+1},\ldots)$ and $X_b=(c_1,\ldots,c_{i-1},b,c_{i+1},\ldots)$ in $\mathrm{dom}(f)$ that differ only in their $i$-th position, such that $f(X_a)\neq f(X_b)$.
We denote the set of variables upon which $f$ depends by $\mathcal{V}(f)$. Intuitively, $f$ depends upon $x_i$ if $x_i$ ``appears" in $f$ and cannot be canceled. For example, if $f=x_4+x_1x_5-x_7$, then $\mathcal{V}(f)=\{x_1,x_4,x_5,x_7\}$, and if $f=(x_1^2x_2)/(x_2x_3)$, then $\mathcal{V}(f)=\{x_1,x_3\}$. Note that if $f$ is a constant function, then $\mathcal{V}(f)=\varnothing$.

We shall say that two generalized multivariate functions $f$ and $g$ are \textit{identical}, written $f\equiv g$, if $\mathcal{V}(f)=\mathcal{V}(g)$ and $f(X)=g(X)$ for every $X\in\mathrm{dom}(f)\cap\mathrm{dom}(g)$. For example, $f=x_1/(x_2/x_3)$ and $g=(x_1x_3)/x_2$ are identical. We stress that, in order for two multivariate functions to be identical, it is not required that they have the same domain (as is the case for \textit{equal} functions), but rather that their values be equal on the domain shared by both. For example, $f=(x_1x_2)/x_1$ and $g=x_2$ are identical, although $\mathrm{dom}(f)=\{(a_1,a_2,\ldots)\in\widehat{\mathds{R}}^{\infty}:a_1\neq0\ \hbox{and}\ a_1\neq\infty\}$ and $\mathrm{dom}(g)=\widehat{\mathds{R}}^{\infty}$. In particular, note that by saying $f\equiv0$, we mean that $f(X)=0$ for every $X\in\mathrm{dom}(f)$.

As we shall see, verifying the dependency of a generalized multivariate function $f$ on a specific variable will be very important in our analysis. For generalized multivariate \textit{rational} functions, this can be done via the partial derivatives of $f$. Given a tuple $(c_1,c_2,\ldots)\in\mathrm{dom}(f)$, we define the \textit{partial derivative of $f$ with respect to $x_i$ at} $(c_1,c_2,\ldots)$ to be
\[
\frac{\partial f}{\partial x_i}(c_1,c_2,\ldots)=\lim_{t\to0}\frac{f(c_1,\ldots,c_{i-1},c_i+t,c_{i+1},\ldots)-f(c_1,\ldots,c_{i-1},c_i,c_{i+1},\ldots)}{t}.
\]
If $x_i\notin\mathcal{V}(f)$, then $f(c_1,\ldots,c_{i-1},c_i+t,c_{i+1},\ldots)=f(c_1,\ldots,c_{i-1},c_i,c_{i+1},\ldots)$ for every $t$ such that $(c_1,\ldots,c_{i-1},c_i+t,c_{i+1},\ldots)\in\mathrm{dom}(f)$. Thus, $\frac{\partial f}{\partial x_i}(c_1,c_2,\ldots)=0$. Therefore, if $x_i\notin\mathcal{V}(f)$, then $\partial f/\partial x_i\equiv0$. In general, the converse is not always true, see~\cite[p. 79]{Zorich}, but it holds if $f$ is rational. In order to verify this, suppose that $x_i\in\mathcal{V}(f)$. Hence, there exist two tuples $X_a=(c_1,\ldots,c_{i-1},a,c_{i+1},\ldots)$ and $X_b=(c_1,\ldots,c_{i-1},b,c_{i+1},\ldots)$ in $\mathrm{dom}(f)$ that differ only in their $i$-th position, such that $f(X_a)\neq f(X_b)$. Consider the single-variable function
\[
\psi(t)=f(c_1,\ldots,c_{i-1},t,c_{i+1},\ldots).
\]
Since $f$ is a rational function, it follows that $\psi(t)$ is rational. Hence, it can be prolonged to a continuous function from $\widehat{\mathds{R}}$ to $\widehat{\mathds{R}}$ that has a continuous derivative on $\widehat{\mathds{R}}$. Hence, we may apply the Mean Value Theorem to $\psi(t)$. Thus, there exists $t_0\in\widehat{\mathds{R}}$ such that
\[
\frac{\psi(a)-\psi(b)}{a-b}
=\psi'(t_0)=\frac{\partial f}{x_i}(c_1,\ldots,c_{i-1},t_0,c_{i+1},\ldots).
\]
Since $\psi(a)-\psi(b)=f(X_a)-f(X_b)\neq0$, it follows that $\partial f/\partial x_i\not\equiv 0$, as required. To conclude: If $f$ is a generalized \textit{rational} multivariate function, then $x_i\in\mathcal{V}(f)$ if and only if $\partial f/\partial x_i\not\equiv 0$. Note that $\mathcal{V}(\partial f/\partial x_i)\subseteq\mathcal{V}(f)$. Indeed, if $x_j\in\mathcal{V}(\partial f/\partial x_i)$, then $\mathcal{V}(\partial f/\partial x_i)\neq\varnothing$. Thus, $\partial f/\partial x_i\not\equiv0$, so $x_i\in\mathcal{V}(f)$, as claimed.
\subsection{Arithmetic Expressions}
Let $f$ be a non-constant generalized multivariate function. We shall say that $f$ is an \textit{arithmetic expression} (abbreviated as AE) if $f$ is either an atomic function, $f(x_1,x_2,x_3,\ldots)=x_i$, where $i\geqslant 1$, or if $f$ is of the form $f\equiv g\ast h$, where $\ast\in\{+,-,\times,\div\}$ and $g,h$ are AE's such that
$\mathcal{V}(g)\cupdot\mathcal{V}(h)=\mathcal{V}(f)$. (Generally, we shall use the notation $A\cupdot B$ to indicate that the sets $A$ and $B$, which constitute the union, are disjoint.)

The set of all AE's will be denoted by $\mathcal{A}$, and the subsets of all AE's $f$ satisfying $\mathcal{V}(f)=\{x_1,x_2,\ldots,x_n\}$ will be denoted by $\mathcal{A}_n$. How many non-identical AE's are there in $\mathcal{A}_n$? Clearly, $x_1$ is the only non-identical AE in $\mathcal{A}_1$. Furthermore, as one can verify, there are $6$ non-identical AE's in $\mathcal{A}_2$, namely:
\[
x_1+x_2\qquad x_1-x_2\qquad x_2-x_1\qquad x_1x_2\qquad \frac{x_1}{x_2}\qquad \frac{x_2}{x_1}
\]
and $68$ non-identical AE's in $\mathcal{A}_3$, namely:
\[
\begin{array}{ccccc}
  x_1+x_2+x_3  & x_1+x_2-x_3    & x_1+x_3-x_2   & x_2+x_3-x_1    & x_1-x_2-x_3 \\
  x_3-x_1-x_2    & x_2-x_3-x_1   & x_1x_2x_3     & x_1x_2+x_3 & x_1x_3+x_2  \\
  x_2x_3+x_1     & x_1x_2-x_3    & x_1x_3-x_2    & x_2x_3-x_1  & x_3-x_1x_2  \\
  x_2-x_1x_3    & x_1-x_2x_3    & x_1(x_2+x_3)   & x_2(x_1+x_3) & x_3(x_1+x_2) \\
  x_1(x_2-x_3)  & x_1(x_3-x_2)  & x_2(x_1-x_3) & x_2(x_3-x_1) & x_3(x_1-x_2) \\
  x_3(x_2-x_1)  & x_1/x_2+x_3   & x_2/x_1+x_3 & x_1/x_3+x_2 & x_3/x_1+x_2 \\
  x_2/x_3+x_1   & x_3/x_2+x_1   & x_1/x_2-x_3  & x_2/x_1-x_3 & x_1/x_3-x_2 \\
  x_3/x_1-x_2   & x_2/x_3-x_1   & x_3/x_2-x_1  & x_3-x_1/x_2 & x_3-x_2/x_1 \\
  x_2-x_1/x_3   & x_2-x_3/x_1   & x_1-x_2/x_3  & x_1-x_3/x_2 & x_1/(x_2+x_3)\\
  x_2/(x_1+x_3) & x_3/(x_1+x_2) & x_1/(x_2-x_3) & x_1/(x_3-x_2) & x_2/(x_1-x_3) \\
  x_2/(x_3-x_1) & x_3/(x_1-x_2) & x_3/(x_2-x_1) & (x_2+x_3)/x_1 & (x_1+x_3)/x_2\\
  (x_1+x_2)/x_3 & (x_2-x_3)/x_1 & (x_3-x_2)/x_1 & (x_1-x_3)/x_2 & (x_3-x_1)/x_2\\
  (x_1-x_2)/x_3 & (x_2-x_1)/x_3 & (x_1x_2)/x_3 & (x_1x_3)/x_2 & (x_2x_3)/x_1 \\
  x_1/(x_2x_3) & x_2/(x_1x_3)  & x_3/(x_1x_2)
\end{array}
\]
It can be further shown that $|\mathcal{A}_4|=1170$, $|\mathcal{A}_5|=27142$, and $|\mathcal{A}_6|=793002$.
For additional values, see~\cite{Sloane_A140606}.
\subsection{Isomorphic Arithmetic Expressions}
Let $X=\{x_1,x_2,x_3,\ldots\}$. For every $f\in \mathcal{A}$ and for every permutation $\sigma\in\mathrm{Sym}(X)$, we define $\sigma\cdot f$ to be the generalized multivariate function obtained by permuting the variables of $f$ according to the permutation $\sigma$, that is,
\[
(\sigma\cdot f)(X)\coloneqq f(\sigma\cdot X),
\]
where $\sigma\cdot X$ denotes the tuple $\sigma\cdot (x_1,x_2,\ldots)\coloneqq(\sigma(x_1),\sigma(x_2),\ldots)$. For example, by applying the permutation $\sigma=(x_1\ x_2\ x_3\,\ldots)$, defined by $x_i\mapsto x_{i+1}$, to $f=x_3+x_2x_7$, we get $\sigma\cdot f=x_4+x_3x_8$. Clearly, $\sigma\cdot f\in \mathcal{A}$ for every $\sigma\in\mathrm{Sym}(X)$. In addition, it is straightforward to verify that $\varepsilon\cdot f\equiv f$, where $\varepsilon$ denotes the trivial permutation of $X$, and that $(\sigma\tau)\cdot f\equiv\sigma\cdot(\tau\cdot f)$ for every $\sigma,\tau\in\mathrm{Sym}(X)$. In other words, the symmetric group $\mathrm{Sym}(X)$ acts on $\mathcal{A}$ via permutations.

Let $f,g\in\mathcal{A}$. We shall say that $f$ is \textit{isomorphic} to $g$, written $f\cong g$, if there exists $\sigma\in\mathrm{Sym}(X)$ such that $\sigma\cdot f\equiv g$. For example, $f=(x_1+x_2)(x_3-x_4)$ and $g=(x_4-x_1)(x_5+x_3)$ are isomorphic. Indeed, if $\sigma=(x_1\ x_3\ x_4)(x_2\ x_5)$, then $\sigma\cdot f\equiv g$.

Recall that, in general, if $\sim$ is an equivalence relation over a set $S$ and if $A\subseteq S$, then the \textit{quotient set of $A$ by $\sim$}, written $\overline{A}$, is the set of all equivalence classes $[a]$ of elements $a$ in $A$, that is, $\overline{A}=\{[a]:a\in A\}$. It is straightforward to verify that $\cong$ is an equivalence relation over $\mathcal{A}$, and hence over $\mathcal{A}_n$. The goal of this paper is to count the number of equivalence classes of $\mathcal{A}_n$ under $\cong$. In other words, we shall be interested in finding $|\overline{\mathcal{A}_n}|$. As can be shown,
\[
\overline{\mathcal{A}_1}=\{\,[x_1]\,\}
\]
\[
\overline{\mathcal{A}_2}=\Big\{\,[x_1+x_2]\,,\,[x_1-x_2]\,,\,[x_1x_2]\,,\,[\textstyle\frac{x_1}{x_2}]\,\Big\}
\]
and
\[
\overline{\mathcal{A}_3}=\left\{
\begin{array}{c}
 [x_1+x_2+x_3]\,,\,[x_1+x_2-x_3]\,,\,[x_1-x_2-x_3]\,,\,[x_1x_2+x_3]\,,\, \\[0.15cm]
 [x_1x_2-x_3]\,,\,[x_1-x_2x_3]\,,\,[x_1(x_2+x_3)]\,,\,[x_1(x_2-x_3)]\,,\,\\[0.15cm]
 [\frac{x_1+x_2}{x_3}]\,,\,[\frac{x_1-x_2}{x_3}]\,,\,[\frac{x_1}{x_2-x_3}]\,,\,[\frac{x_1}{x_2+x_3}]\,,\,[\frac{x_1}{x_2}-x_3]\,,\, \\[0.15cm]
 [\frac{x_1}{x_2}+x_3]\,,\,[x_1-\frac{x_2}{x_3}]\,,\,[x_1x_2x_3]\,,\,[\frac{x_1}{x_2x_3}]\,,\,[\frac{x_1x_2}{x_3}]
\end{array}\right\}.
\]
Hence, $|\overline{\mathcal{A}_1}|=1$, $|\overline{\mathcal{A}_2}|=4$, and $|\overline{\mathcal{A}_3}|=18$.

Note that if $f\in\mathcal{A}$, then there exists $\sigma\in\mathrm{Sym}(X)$ such that $\sigma\cdot f\in\mathcal{A}_n$. It follows that
\[
\overline{\mathcal{A}}=\overline{\mathcal{A}_1}\cupdot\overline{\mathcal{A}_2}\cupdot\overline{\mathcal{A}_3}\cupdot\cdots
\]
This fact allows us to define a $\ast$-like operation on the equivalence classes of $\mathcal{A}$ as follows: Let $\ast\in\{+,-,\times,\div\}$ and let $[f],[g]\in\overline{\mathcal{A}}$. If $\sigma,\tau$ are permutations of $X$ such that $\sigma\cdot f\in\mathcal{A}_n$ and $\tau\cdot g\in\mathcal{A}_m$, then we define
\[
[f]\ast[g]\coloneqq\big[(\sigma\cdot f)\ast(\lambda\tau\cdot g)\big],
\]
where $\lambda$ is the permutation $\lambda=(x_1\ x_{n+1})(x_2\ x_{n+2})\cdots(x_m\ x_{n+m})$. For example,
\[
\textstyle\left[x_7x_5+x_2\right]\left[\frac{x_1x_8}{x_6}-x_2\right]=\left[(x_1x_2+x_3)(\frac{x_4x_5}{x_7}-x_6)\right]
\]
and
\[
\textstyle\left[x_3x_5\right]-\left[x_3x_5\right]=\left[x_1x_2-x_3x_4\right].
\]
As we shall prove later, this operation is well-defined. It is worthwhile noting that this $\ast$-like operation is compatible with the usual $\ast$ operation between AE's, in the sense that if $f,g\in\mathcal{A}$ satisfy $\mathcal{V}(f)\cap\mathcal{V}(g)=\varnothing$, then $[f\ast g]=[f]\ast[g]$. The proof will be given in Proposition~1 of Appendix~B.
\subsection{The three types of arithmetic expressions}
The \textit{negation} of a generalized multivariate function $f$, written $-f$, is the function defined by
$(-f)(X)\coloneqq-f(X)$. Note that if $f\in\mathcal{A}$, then it may happen that $-f\notin\mathcal{A}$. For example, if $f=x_1x_2$, then $f\in\mathcal{A}$ but $-f\notin\mathcal{A}$. If $-f\in\mathcal{A}$ and satisfies $f\cong -f$, then $f$ is called \textit{self-negative}. For example, $x_1-x_2$ is self-negative since $x_1-x_2\cong x_2-x_1$, but $x_1-x_2x_3$ is not. Another example of a self-negative expression is $x_1+x_2(x_3-x_4)-x_5$. The collection of AE's can be divided into three types as follows: we say that $f\in\mathcal{A}$ is of the \textit{first type} if $-f\notin \mathcal{A}$. If $-f\in\mathcal{A}$ but $f\ncong-f$, then we say that $f$ is of the \textit{second type}. Finally, if $-f\in\mathcal{A}$ and $f\cong-f$, then we say that $f$ is of the \textit{third type}. For example, $x_1+x_2x_3$ is of the first type, $x_1-x_2x_3$ is of the second type, and $x_1(x_2-x_3)$ is of the third type. As we shall see later, the type of an AE is invariant under permutation.
\subsection{Classifying expressions in $\mathcal{A}$}
Beyond the division into three types, $\mathcal{A}$ can be further divided into four main categories. Since these categories are recursive in nature, we shall first refer to AE's $f$ with $|\mathcal{V}(f)|\in\{1,2\}$.

Let $i$ and $j$ be different positive integers. If $f\equiv x_i$ is an atomic expression, then we shall say that $f$ \textit{ends with $\times$}. If $f\equiv x_i+x_j$, then we shall say that $f$ \textit{ends with $+$}. If $f\equiv x_i-x_j$, then we shall say that $f$ \textit{ends with $-$}. If $f\equiv x_i/x_j$, then we shall say that $f$ \textit{ends with $\div$}. Finally, if $f\equiv x_ix_j$, then we shall say that $f$ \textit{ends with $\times$}.

Next, let $f\in\mathcal{A}$ with $|\mathcal{V}(f)|\geqslant3$. We shall say that $f$ \textit{ends with $+$} if there exist $g,h\in\mathcal{A}$ such that
\[
f\equiv g+h\quad\text{and}\quad\mathcal{V}(g)\cupdot\mathcal{V}(h)=\mathcal{V}(f),
\]
where $g$ and $h$ end with either $+$, $\times$, or $\div$. We shall say that $f$ \textit{ends with $\times$} if there exist $g,h\in\mathcal{A}$ such that
\[
f\equiv gh\quad\text{and}\quad\mathcal{V}(g)\cupdot\mathcal{V}(h)=\mathcal{V}(f),
\]
where $g$ and $h$ end with either $-$, $+$, or $\times$. We shall say that $f$ \textit{ends with $\div$} if there exist $g,h\in\mathcal{A}$ such that
\[
f\equiv \frac{g}{h}\quad\text{and}\quad\mathcal{V}(g)\cupdot\mathcal{V}(h)=\mathcal{V}(f),
\]
where $g$ and $h$ end with either $+$, $-$, or $\times$. Finally, we shall say that $f$ \textit{ends with $-$} if there exist $g,h\in\mathcal{A}$ such that
\[
f\equiv g-h\quad\text{and}\quad\mathcal{V}(g)\cupdot\mathcal{V}(h)=\mathcal{V}(f),
\]
where $g$ and $h$ end with either $+$, $\times$, or $\div$, and $h$ is of the first type. For example, by the above definitions, the expression $f\equiv (x_1-x_2)/x_3-x_4-x_5$ ends with $-$, while $f\equiv x_4+x_5-(x_1-x_2)/x_3$ ends with $+$. As another example, by the above definitions, the expression $f\equiv x_1(x_2/x_3-x_4)$ ends with $\times$, while $f\equiv (x_2/x_3)(x_1-x_4)$ ends with $\div$.

Given an operator $\ast\in\{+,-,\times,\div\}$, we shall denote by $\mathcal{A}(\ast)$ the set of AE's in $\mathcal{A}$ that end with $\ast$. We define $\mathcal{A}_n(\ast)$ similarly. Furthermore, if $\ast,\star,\ldots,\bullet\in\{+,-,\times,\div\}$, then we shall denote the union $\mathcal{A}(\ast)\cup\mathcal{A}(\star)\cup\cdots\cup\mathcal{A}(\bullet)$ concisely by $\mathcal{A}(\ast,\star,\ldots,\bullet)$. As we shall prove later, $\mathcal{A}(\ast)$ is invariant under permutation; that is, if $f\in\mathcal{A}(\ast)$, then $\sigma\cdot f\in\mathcal{A}(\ast)$ for every $\sigma\in\mathrm{Sym}(X)$.
\subsection{The operator of zero assigning}
Given $f\in\mathcal{A}$ and $x_i\in\mathcal{V}(f)$, we define the \textit{zero-assigning operator} $f^i$ to be the expression obtained by substituting $x_i=0$ into $f$. As we shall prove later, $f^i$ is well-defined for every $x_i\in\mathcal{V}(f)$. For example, if
\[
f=\frac{x_1+x_2}{x_5-\frac{x_3}{x_4}}
\]
then by setting $x_3=0$, we get the expression $f^3=(x_1+x_2)/x_5$. Note that in this example, $f^3\in\mathcal{A}$ and $\mathcal{V}(f^3)\subset\mathcal{V}(f)$. Generally, we have that $\varnothing\subseteq\mathcal{V}(f^i)\subset\mathcal{V}(f)$ for every $x_i\in\mathcal{V}(f)$, but $f^i$ is not necessarily an AE. For example, if $f=x_1-x_2$, then $f^1=-x_2$, so $f^1\notin\mathcal{A}$. Moreover, $f^i$ can be reduced to be either \textit{identically zero} or \textit{identically infinity}. For example, if  $f=x_1(x_2-x_3/x_4)$, then $f^1\equiv0$ and $f^4\equiv\infty$.
\subsection{The multilinear form of arithmetic expression}
The AE's are special rational functions in the sense that they can be expressed uniquely as a quotient of two multilinear polynomials, where the denominator is monic. Recall that a \textit{multilinear polynomial} is a polynomial in which each monomial is a constant times a product of distinct variables. For example, $f=(x_1+x_4)(x_2-x_3x_5)$ is a multilinear polynomial since
\[
f\equiv x_1x_2-x_1x_3x_5+x_2x_4-x_3x_4x_5.
\]
It will be convenient to consider constant polynomials also as multilinear as well. If the monomials of a multilinear polynomial $P$ are arranged in lexicographic order (e.g., $x_1x_3$ before $x_2x_3x_4$ and $x_2$ before $x_2x_3$), then we say that $P$ is \textit{monic} if the coefficient of the leading monomial, according to this order, is positive. For example, $x_2-x_3x_4$ is monic, while $x_2-x_1x_3$ is not. Notice that if $P$ is not monic, then $-P$ is monic and vice versa.

It is straightforward to show that, given an arbitrary variable $x_i$, every multilinear polynomial $P$ can be uniquely expressed as $P\equiv x_iP'+P''$, where $P'$ and $P''$ are multilinear polynomials (possibly constant) such that $x_i\notin\mathcal{V}(P')\cup\mathcal{V}(P'')$. Note that in this presentation, $x_i\in\mathcal{V}(P)$ if and only if $\partial P/\partial x_i\equiv P'\not\equiv0$. For example, with respect to the variable $x_2$, $P=x_2x_3x_5+x_2x_4$ can be expressed as $P\equiv x_2P'+P''$, where $P'=x_3x_5+x_4$ and $P''=0$, while with respect to the variable $x_3$, $P\equiv x_3P'+P''$, where $P'=x_2x_5$ and $P''=x_2x_4$.

Back to AE's, the fact that every $f\in\mathcal{A}$ can be expressed uniquely as a quotient of two multilinear polynomials is proven in Proposition~A.9. We shall refer to this unique presentation as
\textit{the multilinear form of} $f$. For example, the multilinear form of
\[
f=\frac{x_1-x_2}{\frac{x_3}{x_4}+x_5}+\frac{x_6}{x_7}
\]
is
\[
f\equiv\frac{x_1x_4x_7-x_2x_4x_7+x_3x_6+x_4x_5x_6}{x_3x_7+x_4x_5x_7}.
\]
Note that in the multilinear form, the denominator must be monic. For example, the multilinear form of
\[
f\equiv \frac{x_3-x_2}{x_5x_4-x_1}
\]
is
\[
f\equiv\frac{x_2-x_3}{x_1-x_4x_5}.
\]
If $f\in\mathcal{A}$ and $f\equiv F_1/F_2$ is its multilinear form, then we shall say that $f$ is \textit{monic} if $F_1$ is monic as well.
\section{The Fundamental Theorem of Arithmetic expressions}
As mentioned above, the goal of this paper is to count the number of non-isomorphic AE's in $\mathcal{A}_n$. As the definition implies, this should be done recursively, by constructing every AE in $\mathcal{A}$ using other AE's that consist of a fewer number of variables. For example, every AE with $4$ variables can be constructed by either adding, subtracting, multiplying, or dividing other AE's with $1$, $2$, or $3$ variables. Naturally, one must ask: Does this kind of construction involve double counting? In other words, how do we ensure that the AE's we get are non-isomorphic, and how do we ensure that in this process we do not create the same AE more than once? To illustrate this difficulty, let us consider $f=x_1+x_2-x_3-x_4$. In this case, $f$ can be obtained in more than one way from AE's with a fewer number of variables. For instance, it can be obtained either by adding $x_1-x_3$ and $x_2-x_4$ or by subtracting $x_3+x_4$ from $x_1+x_2$. As a start, in order to avoid such duplication so that we can perform a correct enumeration, we must verify that every AE ``falls" into exactly one of the subsets $\mathcal{A}(+)$, $\mathcal{A}(-)$, $\mathcal{A}(\times)$, and $\mathcal{A}(\div)$. In other words, we verify that every AE ends with exactly either $+$, $-$, $\times$, or $\div$. After establishing this fact, we proceed with finding a sort of unique ``canonical" forms that characterizes the AE's in each of the subsets $\mathcal{A}(+)$, $\mathcal{A}(-)$, $\mathcal{A}(\times)$, and $\mathcal{A}(\div)$. These ``canonical" forms are described in our following crucial theorem:
\begin{customthm}{}[The Fundamental Theorem of Arithmetic Expressions]
Every $f\in\mathcal{A}$ ends with exactly one of the operators in $\{+,-,\times,\div\}$; that is,
\[
\mathcal{A}=\mathcal{A}(+)\cupdot\mathcal{A}(-)\cupdot\mathcal{A}(\times)\cupdot\mathcal{A}(\div).
\]
Furthermore,
\begin{enumerate}
  \item[\textup{(a)}]Every $f\in \mathcal{A}(+)$ can be expressed uniquely as a sum $f\equiv f_1+\cdots+f_n$ up to the order of the summands, where $n\geqslant2$, $\mathcal{V}(f_1)\cupdot\cdots\cupdot\mathcal{V}(f_n)=\mathcal{V}(f)$, and $f_j\in\mathcal{A}(\times,\div)$ for each $1\leqslant j\leqslant n$. Moreover,

      \textup{(i)} $f$ is of the first type if and only if each one of the $f_j$'s is of the first type.

      \textup{(ii)} $f$ is of the third type if and only if each $f_j$ is of either the second or third type, and those $f_j$'s that are of the second type can be grouped into pairs $f_{j'}+f_{j''}$ where each such pair satisfies $f_{j'}\cong-f_{j''}$.

  \item[\textup{(b)}]Every $f\in \mathcal{A}(-)$ can be expressed uniquely as a difference $f\equiv g-h$, where $g,h\in\mathcal{A}(+,\times,\div)$, $\mathcal{V}(g)\cupdot\mathcal{V}(h)=\mathcal{V}(f)$, and where $h$ is of the first type.

      Moreover, $f$ is of the third type if and only if either $g\cong h$ or $g\cong G+h$, where $G\in\mathcal{A}(+,\times,\div)$, $G$ is of the third type, and it satisfies $\mathcal{V}(G)\cap\mathcal{V}(h)=\varnothing$.

  \item[\textup{(c)}]Every $f\in\mathcal{A}(\times)$ can be expressed uniquely as a product $f\equiv \varepsilon f_1\cdots f_n$ up to the order of the factors, where $n\geqslant1$, $\varepsilon\in\{1,-1\}$, $\mathcal{V}(f_1)\cupdot\cdots\cupdot\mathcal{V}(f_n)=\mathcal{V}(f)$, and each $f_j$ is monic and is either an atomic expression or in $\mathcal{A}(+,-)$. Moreover,

      \textup{(i)} $f$ is of the first type if and only if $\varepsilon=1$ and each one of the $f_j$'s is of the first type.

      \textup{(ii)} $f$ is of the second type if and only if at least one of the $f_j$'s is of the second type and the rest are of either the first or second type.

      \textup{(iii)} $f$ is of the third type if and only if $\varepsilon=1$ and at least one of the $f_j$'s is of the third type.

  \item[\textup{(d)}]Every $f\in \mathcal{A}(\div)$ can be expressed uniquely as a quotient $f\equiv \varepsilon(g/h)$, where $\varepsilon\in\{1,-1\}$, both $g,h\in\mathcal{A}(+,-,\times)$ are monic, and $\mathcal{V}(g)\cupdot\mathcal{V}(h)=\mathcal{V}(f)$. Moreover,

      \textup{(i)} $f$ is of the first type if and only if $\varepsilon=1$ and both $g$ and $h$ are of the first type.

      \textup{(ii)} $f$ is of the second type if and only if at least one of $g$ or $h$ is of the second type and the other is of either the first or second type.

      \textup{(iii)} $f$ is of the third type if and only if $\varepsilon=1$ and at least one of $g$ or $h$ is of the third type.
\end{enumerate}
\end{customthm}
To illustrate the Fundamental Theorem, consider the following AE:
\[
f=\frac{x_1-x_2}{x_3x_4}+x_5-x_6+\frac{x_7}{x_8}-\frac{x_9}{x_{10}}.
\]
In this case $f\in\mathcal{A}(-)$ and is of the third type. Indeed, first note that $f$ can be rearranged as
\[
f\equiv\underbrace{\overbrace{\frac{x_1-x_2}{x_3x_4}}^{G}+\overbrace{\left(x_5+\frac{x_7}{x_8}\right)}^{h_1}}_g
-\underbrace{\left(x_6+\frac{x_9}{x_{10}}\right)}_{h}.
\]
Since $g,h\in\mathcal{A}(+)$ and $h$ is of the first type, it follows by Part~(b) of the Fundamental Theorem that $f\in\mathcal{A}(-)$. In addition, observe that $G$ is of the third type and $h_1\cong h$. Thus, $g\cong G+h$, so $f$ is of the third type, as claimed. Let us consider another example:
\[
f=\overbrace{\frac{x_1}{x_2-x_3}}^{f_1}+\overbrace{\frac{x_4-x_5x_6}{x_7}}^{f_2}+\overbrace{\frac{x_8x_9-x_{10}}{x_{11}}}^{f_3}
\]
In this case $f\in\mathcal{A}(+)$ and is of the third type. Indeed, first note that $f_1,f_2,f_3\in\mathcal{A}(\div)$, so $f\in\mathcal{A}(+)$ by Part~(a) of the Fundamental Theorem. In addition, since $f_1$ is of the third type, $f_2,f_3$ are of the second type and $f_2\cong-f_3$, it follows by Part~(a.ii) of the Fundamental Theorem that $f$ is of the third type.

Now, given the canonical forms that the Fundamental Theorem describes, and given that two isomorphic AE's differ by some permutation of their variables, the process of constructing all the non-isomorphic AE's is nearly complete. As an example, let us describe the classification of all non-isomorphic AE's with $1$, $2$, and $3$ variables:
\[
\begin{array}{c||c|c|c|c|}
                & \mathcal{A}_1(+) & \mathcal{A}_1(-) & \mathcal{A}_1(\times) & \mathcal{A}_1(\div) \\ \hline\hline
  \hbox{Type 1} &                  &                  &  x_1                  &\\ \hline
  \hbox{Type 2} &                  &                  &                       &\\ \hline
  \hbox{Type 3} &                  &                  &                       &\\ \hline
\end{array}
\]

\[
\begin{array}{c||c|c|c|c|}
                & \mathcal{A}_2(+) & \mathcal{A}_2(-) & \mathcal{A}_2(\times) & \mathcal{A}_2(\div) \\ \hline\hline
  \hbox{Type 1} &   x_1+x_2        &                  &  x_1x_2               & x_1/x_2 \\ \hline
  \hbox{Type 2} &                  &                  &                       &\\ \hline
  \hbox{Type 3} &                  & x_1-x_2          &                       &\\ \hline
\end{array}
\]

\[
\begin{array}{c||c|c|c|c|}
              & \mathcal{A}_3(+) & \mathcal{A}_3(-) & \mathcal{A}_3(\times) & \mathcal{A}_3(\div) \\ \hline\hline
\hbox{Type 1} & x_1+x_2+x_3      &                  &  x_1x_2x_3   &  x_1/(x_2x_3) \\
              & x_1+x_2x_3       &                  & x_1(x_2+x_3) &  (x_2x_3)/x_1 \\
              & x_1+x_2/x_3      &                  &              &  x_1/(x_2+x_3)\\
              &                  &                  &              &  (x_2+x_3)/x_1 \\ \hline
\hbox{Type 2} &                  & x_1-(x_2+x_3)    &              & \\
              &                  & (x_2+x_3)-x_1    &              & \\
              &                  & x_1-x_2x_3       &              & \\
              &                  & x_2x_3-x_1       &              & \\
              &                  & x_1-x_2/x_3      &              & \\
              &                  & x_2/x_3-x_1      &              & \\ \hline
\hbox{Type 3} &                  &                  &  x_1(x_2-x_3)& (x_2-x_3)/x_1\\
              &                  &                  &              & x_1/(x_2-x_3)\\ \hline
\end{array}
\]

The proof of the Fundamental Theorem requires quite a bit of effort. For this reason, in order to aid in readability and focus on our main results regarding the combinatorial aspect of the paper, we have postponed the proof of the Fundamental Theorem to Appendix~A.
\section{The enumeration of arithmetic expressions}
In this section, we shall prove the main result of this paper, Theorem~4.3, which presents the formulas for calculating the number of non-isomorphic arithmetic expressions with $n$ variables. In order to prove Theorem~4.3, we shall assume the correctness of the Fundamental Theorem.
\subsection{Preliminaries: partitions and multisets}
Recall that a \textit{partition} of a positive integer $n$ is a representation of $n$ as a sum of positive integers. For example, the seven partitions of $5$ are
\[
\begin{array}{l}
1+1+1+1+1\\
1+1+1+2\\
1+2+2\\
1+1+3\\
1+4\\
2+3\\
5\\
\end{array}
\]
The partition $n$, with only one component, is called the \textit{trivial partition}. One way to describe a partition of $n$ is as a tuple of positive integers $(k_1,k_2,\ldots,k_r)$ such that $k_1\leqslant k_2\leqslant\cdots\leqslant k_r$ and $k_1+k_2+\cdots+k_r=n$. In certain cases, where we want to emphasize the number of repetitions of summands in the partition, we will use the following compact form: $(k_1^{m_1},k_2^{m_2},\ldots,k_s^{m_s})$. Here, $m_1$ is the number of $k_1$'s, $m_2$ is the number of $k_2$'s, etc. For example, in this notation, the partitions of $5$ can be written as follows:
\[
(1^5)\ ,\ (1^3,2^1)\ ,\ (1^1,2^2)\ ,\ (1^2,3^1)\ ,\ (1^1,4^1)\ ,\ (2^1,3^1)\ ,\ (5^1).
\]
To indicate that $\lambda$ is a partition of $n$, we shall use the notation $\lambda\vdash n$. For example, $(1^1,2^2)\vdash 5$.

Recall that a \textit{multiset} is a collection of objects in which elements may
occur more than once, including an infinite number of times. The number of times that an object occurs in a multiset is
called its \textit{repetition number}. It is customary to use the notation $r\cdot x$ to mean that the repetition number of $x$ is $r$. For example, the multiset $\{a,a,a,b,b,c,c,c,\ldots\}$ can be denoted by $\{3\cdot a,2\cdot b,\infty\cdot c\}$. If all repetition numbers are $\infty$, then it is well known that the number of multisubsets of $\{\infty\cdot1,\infty\cdot2,\ldots,\infty\cdot n\}$ with $k$ elements is $\multiset{n}{k}\coloneqq\binom{n+k-1}{k}$ (see~\cite[p. 38]{Chen}). For example, the number of multisubsets of $\{\infty\cdot1,\infty\cdot2,\infty\cdot3\}$ with $3$ elements is $\multiset{3}{3}=\binom{5}{3}=10$. Indeed, these multisubsets are
\[
\begin{gathered}
\{1,1,1\}\quad\{1,1,2\}\quad\{1,1,3\}\quad\{1,2,2\}\quad\{1,2,3\}\\
\{1,3,3\}\quad\{2,2,2\}\quad\{2,2,3\}\quad\{2,3,3\}\quad\{3,3,3\}\\
\end{gathered}
\]
In other words, $\multiset{n}{k}$ is the number of ways to choose $k$ elements from $n$ distinct elements, where we may choose the same element more than once. The following lemma will play an important role in our study.
\begin{lemma}[The colored weights problem]
Let $A$ be a set of weights such that for every positive integer $k$, we have $N(k)$ differently-colored weights, each
weighing $k$ kilograms. Then, the number of ways to use these weights to create a total weight of $n$ kilograms, where it is allowed to use the same weight more than once, is
\[
\omega_n(A)\coloneqq\sum_{(k_1^{m_1},\ldots,k_s^{m_s})\vdash n}\dispmultiset{N(k_1)}{m_1}\dispmultiset{N(k_2)}{m_2}
\cdots\dispmultiset{N(k_s)}{m_s}.
\]
\end{lemma}
\begin{proof}
Note that every partition $(k_1^{m_1},\ldots,k_s^{m_s})$ of $n$ corresponds exactly to one uncolored weighing option of $n$ kilograms, in which, for every $j$, we use $m_j$ (uncolored) weights of $k_j$ kilograms each. The total number of colored weighings of $n$ kilograms, is the number of ways to color every such uncolored weighing option. This is achieved by coloring the weights of the same weight separately as follows: to color the $m_j$ weights of $k_j$ kilograms, we choose $m_j$ colors from the available $N(k_j)$ colors for the weight $k_j$. Since we allow the use of  the same color more than once, the number of ways to color the $m_j$ weights of $k_j$ kilograms is $\multiset{N(k_j)}{m_j}$. Hence, the number of ways to color the weights that correspond to the partition $(k_1^{m_1},\ldots,k_s^{m_s})$ is
\[
\dispmultiset{N(k_1)}{m_1}\dispmultiset{N(k_2)}{m_2}
\cdots\dispmultiset{N(k_s)}{m_s}
\]
By running through all partitions of $n$, it follows that the total number of colored weighing of $n$ kilograms is $\omega_n(A)$, as claimed.
\end{proof}

As an illustrative example, suppose that the set of weights is
\begin{center}
\begin{tikzpicture}
        \filldraw[color=black, fill=blue!40!white](-0.3,-0.3) rectangle (0.3,0.3);
        \node at (0,0) {$1$};
        \filldraw[color=black, fill=red!40!white](0.4,-0.3) rectangle (1,0.3);
        \node at (0.7,0) {$1$};
        \filldraw[color=black, fill=blue!40!white] (1.1,-0.3)--(1.7,-0.3)--(1.4,0.3)--(1.1,-0.3);
        \node at (1.4,-0.1) {$2$};
        \filldraw[color=black,fill=red!40!white](1.8,-0.3)--(2.4,-0.3)--(2.1,0.3)--(1.8,-0.3);
        \node at (2.1,-0.1) {$2$};
        \filldraw[color=black, fill=yellow!40!white]
                         (2.5,-0.3)--(3.1,-0.3)--(2.8,0.3)--(2.5,-0.3);
        \node at (2.8,-0.1) {$2$};
        \filldraw[color=black, fill=blue!40!white] (3.5,0) circle(0.3);
        \node at (3.5,0) {$4$};
        \node at (-0.9,0) {$A=\Big\{$};
        \node at (4,0) {$\Big\}$};
\end{tikzpicture}
\end{center}
Thus, in this set, $N(1)=2$, $N(2)=3$, $N(3)=0$, $N(4)=1$, and $N(k)=0$ for all $k\geqslant5$. Let us find the number of ways to create the weight $n=4$. The partitions of $4$ are $(1^4),(1^2,2^1),(1^1,3^1),(2^2),(4^1)$. Thus, according to Lemma~4.1, the number of ways to create the weight $n=4$ is
\[
\begin{split}
  \omega_4(A)
  &=\dispmultiset{N(1)}{4}+\dispmultiset{N(1)}{2}\hspace{-0.15cm}\dispmultiset{N(2)}{1}
  +\dispmultiset{N(1)}{1}\hspace{-0.15cm}\dispmultiset{N(3)}{1}+\dispmultiset{N(2)}{2}+\dispmultiset{N(4)}{1}\\
  &=\binom{5}{4}+\binom{3}{2}\binom{3}{1}+\binom{2}{1}\binom{0}{1}+\binom{4}{2}+\binom{1}{1}=21.
\end{split}
\]
Below, we present some illustrated examples of such $4$-kg weighings.
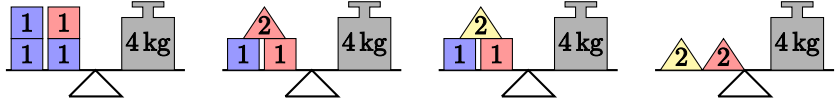
\begin{center}
\begin{tikzpicture}[scale=0.7]
\draw[thick] (-0.5,0)--(0.5,0)--(0,0.5)--(-0.5,0);
\draw[thick] (-1.7,0.5)--(1.7,0.5);
\filldraw[color=black, fill=blue!40!white] (-1.6,0.5) rectangle (-1,1.1);
\filldraw[color=black, fill=blue!40!white] (-0.9,0.5) rectangle (-0.3,1.1);
\filldraw[color=black, fill=blue!40!white] (-1.6,1.1) rectangle (-1,1.7);
\filldraw[color=black, fill=red!40!white] (-0.9,1.1) rectangle (-0.3,1.7);
\filldraw[color=black, fill=black!30!white] (0.5,0.5)--(1.5,0.5)--(1.5,1.5)--(1.1,1.5)--(1.1,1.7)--
                                                   (1.3,1.7)--(1.3,1.8)--(0.7,1.8)--(0.7,1.7)--(0.9,1.7)--
                                                   (0.9,1.5)--(0.5,1.5)--(0.5,0.5);
\node at (-1.3,0.8) {$1$};
\node at (-0.6,0.8) {$1$};
\node at (-1.3,1.4) {$1$};
\node at (-0.6,1.4) {$1$};
\node at (1,1) {$4\,\mathrm{kg}$};
\end{tikzpicture}
\quad
\begin{tikzpicture}[scale=0.7]
\draw[thick] (-0.5,0)--(0.5,0)--(0,0.5)--(-0.5,0);
\draw[thick] (-1.7,0.5)--(1.7,0.5);
\filldraw[color=black, fill=blue!40!white] (-1.6,0.5) rectangle (-1,1.1);
\filldraw[color=black, fill=red!40!white] (-0.9,0.5) rectangle (-0.3,1.1);
\filldraw[color=black, fill=red!40!white] (-1.3,1.1)--(-0.5,1.1)--(-0.9,1.7)--(-1.3,1.1);
\filldraw[color=black, fill=black!30!white] (0.5,0.5)--(1.5,0.5)--(1.5,1.5)--(1.1,1.5)--(1.1,1.7)--
                                                   (1.3,1.7)--(1.3,1.8)--(0.7,1.8)--(0.7,1.7)--(0.9,1.7)--
                                                   (0.9,1.5)--(0.5,1.5)--(0.5,0.5);
\node at (-1.3,0.8) {$1$};
\node at (-0.6,0.8) {$1$};
\node at (-0.9,1.35) {$2$};
\node at (1,1) {$4\,\mathrm{kg}$};
\end{tikzpicture}
\quad
\begin{tikzpicture}[scale=0.7]
\draw[thick] (-0.5,0)--(0.5,0)--(0,0.5)--(-0.5,0);
\draw[thick] (-1.7,0.5)--(1.7,0.5);
\filldraw[color=black, fill=blue!40!white] (-1.6,0.5) rectangle (-1,1.1);
\filldraw[color=black, fill=red!40!white] (-0.9,0.5) rectangle (-0.3,1.1);
\filldraw[color=black, fill=yellow!40!white] (-1.3,1.1)--(-0.5,1.1)--(-0.9,1.7)--(-1.3,1.1);
\filldraw[color=black, fill=black!30!white] (0.5,0.5)--(1.5,0.5)--(1.5,1.5)--(1.1,1.5)--(1.1,1.7)--
                                                   (1.3,1.7)--(1.3,1.8)--(0.7,1.8)--(0.7,1.7)--(0.9,1.7)--
                                                   (0.9,1.5)--(0.5,1.5)--(0.5,0.5);
\node at (-1.3,0.8) {$1$};
\node at (-0.6,0.8) {$1$};
\node at (-0.9,1.35) {$2$};
\node at (1,1) {$4\,\mathrm{kg}$};
\end{tikzpicture}
\quad
\begin{tikzpicture}[scale=0.7]
\draw[thick] (-0.5,0)--(0.5,0)--(0,0.5)--(-0.5,0);
\draw[thick] (-1.7,0.5)--(1.7,0.5);
\filldraw[color=black, fill=yellow!40!white] (-1.6,0.5)--(-0.8,0.5)--(-1.2,1.1)--(-1.6,0.5);
\filldraw[color=black, fill=red!40!white] (-0.8,0.5)--(0,0.5)--(-0.4,1.1)--(-0.8,0.5);
\filldraw[color=black, fill=black!30!white] (0.5,0.5)--(1.5,0.5)--(1.5,1.5)--(1.1,1.5)--(1.1,1.7)--
                                                   (1.3,1.7)--(1.3,1.8)--(0.7,1.8)--(0.7,1.7)--(0.9,1.7)--
                                                   (0.9,1.5)--(0.5,1.5)--(0.5,0.5);
\node at (-0.4,0.75) {$2$};
\node at (-1.2,0.75) {$2$};
\node at (1,1) {$4\,\mathrm{kg}$};
\end{tikzpicture}
\end{center}
For $n=0$, we define $\omega_0(A)=1$. We further define $\omega_n'(A)$ to be the same sum as $\omega_n(A)$, excluding the summand which corresponds to the trivial partition $(n^1)$. Note that, since this summand is
$\multiset{\hspace{-0.05cm}N(n)\hspace{-0.05cm}}{1}=N(n)$, it follows that $\omega_n'(A)=\omega_n(A)-N(n)$.
\subsection{Our main result}
Assuming the fundamental theorem, we are almost ready to prove Theorem~4.3, which is our main result. Before that, we need some additional notation. Given an operator $\ast\in\{+,-,\times,\div\}$, we define $\mathcal{F}(\ast)$, $\mathcal{B}(\ast)$, and $\mathcal{C}(\ast)$ to be the subsets of $\mathcal{A}(\ast)$ which consist of AE's of the first, second, and third type, respectively. If $\ast,\star,\ldots,\bullet\in\{+,-,\times,\div\}$, we define $\mathcal{F}(\ast,\star,\ldots,\bullet)\coloneqq \mathcal{F}(\ast)\cup\mathcal{F}(\star)\cup\cdots\cup\mathcal{F}(\bullet)$. Given a positive integer $n$, we define $\mathcal{F}_n(\ast)\coloneqq \mathcal{F}(\ast)\cap\mathcal{A}_n$. In addition, we define $\mathcal{F}_0(\ast)\coloneqq\{0\}$. The same definitions hold for $\mathcal{B}(\ast)$ and $\mathcal{C}(\ast)$. We further define $\mathcal{MB}(\ast)$ to be the subset of $\mathcal{B}(\ast)$ which includes the monic AE's of $\mathcal{B}(\ast)$.

Since we intend to enumerate representatives of the different equivalence classes in $\mathcal{A}_n$, we first need to convert the results of the Fundamental Theorem into the language of equivalence classes; this is done in the next theorem. Since its proof follows almost directly from the Fundamental Theorem, and in order to maintain the flow of this section, we have placed its proof in the Appendix~B.
\begin{theorem}
Let $n\geqslant 2$. Then:
\begin{enumerate}
  \item[\textup{(a)}]\textup{(i)} Every $[f]\in\overline{\mathcal{A}_n(+)}$ can be expressed uniquely, up to the order of the summands, as
  \[
      [f]=[f_1]+\cdots+[f_r],
  \]
  where each $f_j\in\mathcal{A}_{k_j}(\times,\div)$ and $(k_1,\ldots,k_r)$ is a non-trivial partition of $n$.

  \textup{(ii)} Every $[f]\in\overline{\mathcal{F}_n(+)}$ can be expressed uniquely, up to the order of the summands, as
  \[
      [f]=[f_1]+\cdots+[f_r],
  \]
  where each $f_j\in\mathcal{F}_{k_j}(\times,\div)$ and $(k_1,\ldots,k_r)$ is a non-trivial partition of $n$.

  \textup{(iii)} Every $[f]\in\overline{\mathcal{C}_n(+)}$ can be expressed uniquely, up to the order of the summands, in one of the following ways: either as
  \[
  [f]=[f_1]+\cdots+[f_r],
  \]
  where each $f_j\in\mathcal{C}_{k_j}(\times,\div)$ and $(k_1,\ldots,k_r)$ is a non-trivial partition of $n$;
  or as
  \[
  [f]=[g_1]+[-g_1]+\cdots+[g_s]+[-g_s],
  \]
  where each $g_j\in\mathcal{MB}_{m_j}(\times,\div)$ and $(m_1,\ldots,m_s)$ is a partition of $n/2$; or as
\[
  [f]=[f_1]+\cdots+[f_r]+[g_1]+[-g_1]+\cdots+[g_s]+[-g_s],
\]
where each $f_j\in\mathcal{C}_{k_j}(\times,\div)$, each $g_j\in\mathcal{MB}_{m_j}(\times,\div)$, $(m_1,\ldots,m_s)$ is a partition of some $1\leqslant k<\frac{n}{2}$, and $(k_1,\ldots,k_r)$ is a partition of $n-2k$.
  \item[\textup{(b)}]\textup{(i)} Every $[f]\in\overline{\mathcal{A}_n(-)}$ can be expressed uniquely as
  \[
  [f]=[g]-[h],
  \]
  where $g\in\mathcal{A}_k(+,\times,\div)$ and $h\in\mathcal{F}_{n-k}(+,\times,\div)$ for some $1\leqslant k\leqslant n-1$.

  \textup{(ii)} Every $[f]\in\overline{\mathcal{C}_n(-)}$ can be expressed uniquely either as
  \[
  [f]=[h]-[h],
  \]
  where $h\in\mathcal{F}_{\frac{n}{2}}(+,\times,\div)$, or as
  \[
  [f]=[g]+[h]-[h],
  \]
  where $g\in\mathcal{C}_{n-2k}(+,\times,\div)$ and $h\in\mathcal{F}_{k}(+,\times,\div)$ for some $1\leqslant k<\frac{n}{2}$.
  \item[\textup{(c)}]\textup{(i)} Every $[f]\in\overline{\mathcal{F}_n(\times)}$ can be expressed uniquely, up to the order of factors, either as
      \[
      [f]=\overbrace{[x_1][x_1]\cdots[x_1]}^{\hbox{\tiny $n$ \textup{times}}},
      \]
      or as
      \[
      [f]=[f_1][f_2]\cdots[f_r],
      \]
      where each $f_j\in\mathcal{F}_{k_j}(+)$ and $(k_1,\ldots,k_r)$ is a non-trivial partition of $n$, or as
  \[
  [f]=[f_1][f_2]\cdots[f_{r-m}]\overbrace{[x_1]\cdots[x_1]}^{\hbox{\tiny $m$ \textup{times}}},
\]
 where each $f_j\in\mathcal{F}_{k_j}(+)$ and $(k_1,\ldots,k_{r-m})$ is a partition of $n-m$ for some $1\leqslant m\leqslant n-1$.

\textup{(ii)} Every $[f]\in\overline{\mathcal{B}_n(\times)}$ can be expressed uniquely, up to the order of factors, either as
\[
[f]=[\varepsilon f_1]\cdots[f_r],
\]
where $\varepsilon\in\{-1,1\}$, each $f_j\in\mathcal{MB}_{k_j}(+,-)$, and $(k_1,\ldots,k_r)$ is a non-trivial partition of $n$, or as
\[
[f]=[\varepsilon f_1]\cdots[f_r][g],
\]
where $\varepsilon\in\{-1,1\}$, each $f_j\in\mathcal{MB}_{k_j}(+,-)$, $(k_1,\ldots,k_r)$ is a partition of $k$ for some $1\leqslant k\leqslant n-1$, and $g\in\mathcal{F}_{n-k}(+,\times)$.

\textup{(iii)} Every $[f]\in\overline{\mathcal{C}_n(\times)}$ can be expressed uniquely, up to the order of factors,  either as
\[
[f]=[f_1][f_2]\cdots[f_r],
\]
where each $f_j\in\mathcal{C}_{k_j}(+,-)$ and $(k_1,\ldots,k_r)$ is a non-trivial partition of $n$, or as
\[
[f]=[f_1][f_2]\cdots[f_r][g],
\]
where each $f_j\in\mathcal{C}_{k_j}(+,-)$, $(k_1,\ldots,k_r)$ is a partition of $k$ for some $1\leqslant k\leqslant n-1$, and $g\in\mathcal{F}_{n-k}(+,\times)$, or as
\[
[f]=[f_1][f_2]\cdots[f_r][h],
\]
where each $f_j\in\mathcal{C}_{k_j}(+,-)$, $(k_1,\ldots,k_r)$ is a partition of $k$ for some $1\leqslant k\leqslant n-1$, and $h\in\mathcal{MB}_{n-k}(+,-,\times)$.
\item[\textup{(d)}] \textup{(i)} Every $[f]\in\overline{\mathcal{F}_n(\div)}$ can be expressed uniquely as
\[
[f]=\frac{[g]}{[h]},
\]
where $g\in\mathcal{F}_k(+,\times)$ and $h\in\mathcal{F}_{n-k}(+,\times)$ for some $1\leqslant k\leqslant n-1$.

\textup{(ii)} Every $[f]\in\overline{\mathcal{B}_n(\div)}$ can be expressed uniquely either as
\[
[f]=\frac{[\varepsilon g]}{[h]},
\]
where $\varepsilon\in\{1,-1\}$, $g\in\mathcal{MB}_{n-k}(+,-,\times)$, and $h\in\mathcal{F}_k(+,\times)\cup\mathcal{MB}_k(+,-,\times)$ for some $1\leqslant k\leqslant n-1$, or as
\[
[f]=\frac{[g]}{[\varepsilon h]},
\]
where $\varepsilon\in\{1,-1\}$, $g\in\mathcal{F}_k(+,\times)$, and $h\in\mathcal{MB}_{n-k}(+,-,\times)$ for some $1\leqslant k\leqslant n-1$.

\textup{(iii)} Every $[f]\in\overline{\mathcal{C}_n(\div)}$ can be expressed uniquely as
\[
[f]=\frac{[g]}{[h]},
\]
where either $g\in\mathcal{C}_{n-k}(+,-,\times)$ and $h\in\mathcal{F}_k(+,\times)\cup\mathcal{MB}_k(+,-,\times)\cup\mathcal{C}_k(+,-,\times)$, or
$g\in\mathcal{F}_k(+,\times)\cup\mathcal{MB}_k(+,-,\times)$ and $h\in\mathcal{C}_{n-k}(+,-,\times)$ for some $1\leqslant k\leqslant n-1$.
\end{enumerate}
\end{theorem}
With these unique representations established, we proceed to the proof of our main theorem:
\begin{theorem}
Let $n\geqslant 2$. Then:

\

\begin{flushleft}
  $\begin{aligned}
  \textup{(a)}\qquad&|\overline{\mathcal{F}_n(+)}|
            =\omega_n'(\overline{\mathcal{F}(\times,\div)})\\
      \qquad&|\overline{\mathcal{B}_n(+)}|
            =\omega_n'(\overline{\mathcal{A}(\times,\div)})-|\overline{\mathcal{F}_n(+)}|-|\overline{\mathcal{C}_n(+)}|\\
      \qquad&|\overline{\mathcal{C}_n(+)}|
      =\omega_n'(\overline{\mathcal{C}(\times,\div)})+
\sum_{k=1}^{\lfloor \frac{n}{2}\rfloor}\omega_k(\overline{\mathcal{MB}(\times,\div)})\omega_{n-2k}(\overline{\mathcal{C}(\times,\div)})\\
\end{aligned}$
\end{flushleft}
\begin{flushleft}
  $\begin{aligned}
  \textup{(b)}\qquad&|\overline{\mathcal{F}_n(-)}|=0\\
  \qquad&|\overline{\mathcal{B}_n(-)}|=\sum_{k=1}^{n-1}|\overline{\mathcal{A}_k(+,\times,\div)}||\overline{\mathcal{F}_{n-k}(+,\times,\div)}|
    -|\overline{\mathcal{C}_n(-)}|\\
  \qquad&|\overline{\mathcal{C}_n(-)}|=\sum_{k=1}^{\lfloor\frac{n}{2}\rfloor}|\overline{\mathcal{C}_{n-2k}(+,\times,\div)}|
  |\overline{\mathcal{F}_k(+,\times,\div)}|\\
  \end{aligned}$
\end{flushleft}
\begin{flushleft}
  $\begin{aligned}
  \textup{(c)}\qquad&|\overline{\mathcal{F}_n(\times)}|
  =\omega_n'(\overline{\mathcal{F}(+)})+\sum_{k=0}^{n-1}\omega_k(\overline{\mathcal{F}(+)})\\
  \qquad&|\overline{\mathcal{B}_n(\times)}|
  =2\omega_{n}'(\overline{\mathcal{MB}(+,-)})+\sum_{k=1}^{n-1}
  2\omega_k(\overline{\mathcal{MB}(+,-)})|\overline{\mathcal{F}_{n-k}(+,\times)}|\
  \\
  \qquad&|\overline{\mathcal{C}_n(\times)}|
  =\omega_{n}'(\overline{\mathcal{C}(+,-)})+\sum_{k=1}^{n-1}
  \textstyle\omega_k(\overline{\mathcal{C}(+,-)})\big(|\overline{\mathcal{F}_{n-k}(+,\times)}|
  +\frac{1}{2}|\overline{\mathcal{B}_{n-k}(+,-,\times)}|\big)\\
  \end{aligned}$
\end{flushleft}
\begin{flushleft}
  $\begin{aligned}
  \textup{(d)}\qquad&|\overline{\mathcal{F}_n(\div)}|
  =\sum_{k=1}^{n-1}|\overline{\mathcal{F}_k(+,\times)}||\overline{\mathcal{F}_{n-k}(+,\times)}|\\
  \qquad&|\overline{\mathcal{B}_n(\div)}|=\sum_{k=1}^{n-1}|\overline{\mathcal{B}_{n-k}(+,-,\times)}|\textstyle\big(2|\overline{\mathcal{F}_k(+,\times)}|
+\frac{1}{2}|\overline{\mathcal{B}_k(+,-,\times)}|\big)\\
  \qquad&|\overline{\mathcal{C}_n(\div)}|
  =\sum_{k=1}^{n-1}|\mathcal{C}_{n-k}(+,-,\times)|\big(2|\mathcal{F}_k(+,\times)|
+|\mathcal{B}_k(+,-,\times)|+|\mathcal{C}_k(+,-,\times)|\big)
  \end{aligned}$
\end{flushleft}
\end{theorem}
\begin{proof}
(a) By Theorem~4.2(a.i), $|\overline{\mathcal{A}_n(+)}|$ is the number of sums of the form
\[
[f_1]+\cdots+[f_r],
\]
where $f_j\in\mathcal{A}_{k_j}(\times,\div)$ for each $j$ and where $(k_1,\ldots,k_r)$ is a non-trivial partition of $n$. Since
\[
\overline{\mathcal{A}(\times,\div)}=\overline{\mathcal{A}_1(\times,\div)}
\cupdot\overline{\mathcal{A}_2(\times,\div)}\cupdot\overline{\mathcal{A}_3(\times,\div)}\cupdot\cdots,
\]
it follows that $|\overline{\mathcal{A}_n(+)}|$ is the number of non-trivial ways to create the weight $n$ using the ``weights" in $\overline{\mathcal{A}(\times,\div)}$. Thus, by Lemma~4.1,
\[
|\overline{\mathcal{A}_n(+)}|=\omega_n'(\overline{\mathcal{A}(\times,\div)}).
\]
Similarly, by Theorem~4.2(a.ii), $|\overline{\mathcal{F}_n(+)}|=\omega_n'(\overline{\mathcal{F}(\times,\div)})$, so $|\mathcal{B}_n(+)|=\omega_n'(\overline{\mathcal{A}(\times,\div)})-|\overline{\mathcal{F}_n(+)}|
-|\overline{\mathcal{C}_n(+)}|$, as required.

By Theorem~4.2(a.iii), $|\overline{\mathcal{C}_n(+)}|$ is the number of sums which have one of the following forms:
\begin{gather*}
   [f_1]+\cdots+[f_r], \\
  [g_1]+[-g_1]+\cdots+[g_s]+[-g_s], \\
  [f_1]+\cdots+[f_r]+[g_1]+[-g_1]+\cdots+[g_s]+[-g_s].\\
\end{gather*}
In the first form, each $f_j\in\mathcal{C}_{k_j}(\times,\div)$ and $(k_1,\ldots,k_r)$ is a non-trivial partition of $n$. In the second form each $g_j\in\mathcal{MB}_{m_j}(\times,\div)$ and $(m_1,\ldots,m_s)$ is a partition of $n/2$. In the third form $f_j\in\mathcal{C}_{k_j}(\times,\div)$ and $g_j\in\mathcal{MB}_{m_j}(\times,\div)$ for each $j$, where $(m_1,\ldots,m_s)$ is a partition of some $1\leqslant k<\frac{n}{2}$ and $(k_1,\ldots,k_r)$ is a partition of $n-2k$. Hence, the number of AE's of the first form is $\omega_n'(\overline{\mathcal{C}(\times,\div)})$. The number of AE's of the second form is $\omega_{\frac{n}{2}}(\overline{\mathcal{MB}(\times,\div)})$ and the number of AE's of the third form is
\[
\sum_{1\leqslant k<\frac{n}{2}}\omega_k(\overline{\mathcal{MB}(\times,\div)})\omega_{n-2k}(\overline{\mathcal{C}(\times,\div)}).
\]
Note that since $\omega_0(\overline{\mathcal{C}(\times,\div)})=1$, the number of AE's of the last two forms can be combined. Hence, the total number of AE's in $|\overline{\mathcal{C}_n(+)}|$ is
\[
|\overline{\mathcal{C}_n(+)}|=
\omega_n'(\overline{\mathcal{C}(\times,\div)})+
\sum_{k=1}^{\lfloor\frac{n}{2}\rfloor}\omega_k(\overline{\mathcal{MB}(\times,\div)})\omega_{n-2k}(\overline{\mathcal{C}(\times,\div)}),
\]
as required.

\

(b) Clearly, $|\mathcal{F}_n(-)|=0$. By Theorem~4.2(b.i), $|\overline{\mathcal{A}_n(-)}|$ is the number of differences of the form $[g]-[h]$, where $g\in\mathcal{A}_k(+,\times,\div)$, $h\in\mathcal{F}_{n-k}(+,\times,\div)$, and $1\leqslant k\leqslant n-1$. Hence,
\[
|\overline{\mathcal{A}_n(-)}|=\sum_{k=1}^{n-1}|\overline{\mathcal{A}_k(+,\times,\div)}||\overline{\mathcal{F}_{n-k}(+,\times,\div)}|,
\]
so $|\overline{\mathcal{B}_n(-)}|=|\overline{\mathcal{A}_n(-)}|-|\overline{\mathcal{C}_n(-)}|$, as required. Furthermore, by Theorem~4.2(b.ii), $|\overline{\mathcal{C}_n(-)}|$ is the number of differences which have one of the following forms
\[
[h]-[h]\quad\hbox{or}\quad[g]+[h]-[h],
\]
where in the first form $h\in\mathcal{F}_{\frac{n}{2}}(+,\times,\div)$, and in the second form $h\in\mathcal{F}_{k}(+,\times,\div)$, $g\in\mathcal{C}_{n-2k}(+,\times,\div)$ for some $1\leqslant k<\frac{n}{2}$. Recall that $|\mathcal{C}_0(+,\times,\div)|=1$, so combining these two cases yields
\[
|\overline{\mathcal{C}_n(-)}|=\sum_{k=1}^{\lfloor\frac{n}{2}\rfloor}|\overline{\mathcal{C}_{n-2k}(+,\times,\div)}||\overline{\mathcal{F}_{k}(+,\times,\div)}|,
\]
as required.

\

(c) By Theorem~4.2(c.i), $|\overline{\mathcal{F}_n(\times)}|$ is the number of products which have one of the following forms:
\begin{gather*}
   \overbrace{[x_1][x_1]\cdots[x_1]}^{\hbox{\tiny $n$ times}}, \\
  [f_1][f_2]\cdots[f_r], \\
  [f_1][f_2]\cdots[f_{r-m}]\underbrace{[x_1]\cdots[x_1]}_{\hbox{\tiny $m$ \textup{times}}}.
\end{gather*}
In the second form, each $f_j\in\mathcal{F}_{k_j}(+)$ and $(k_1,\ldots,k_r)$ is a non-trivial partition of $n$. In the third form, each $f_j\in\mathcal{F}_{k_j}(+)$ and $(k_1,\ldots,k_{r-m})$ is a partition of $n-m$ for some $1\leqslant m\leqslant n-1$. Hence, the number of AE's of the second form is $\omega_n'(\overline{\mathcal{F}(+)})$ and the number of AE's of the third form is $\omega_{n-m}(\overline{\mathcal{F}(+)})$. Note that since $\omega_0(\overline{\mathcal{F}(+)})=1$, the last result agrees with the sole case where all the factors are atomic. Therefore, the total number of AE's in $\overline{\mathcal{F}_n(\times)}$ is
\[
|\overline{\mathcal{F}_n(\times)}|=\omega_n'(\overline{\mathcal{F}(+)})+\sum_{m=1}^{n}\omega_{n-m}(\overline{\mathcal{F}(+)})=
\omega_n'(\overline{\mathcal{F}(+)})+\sum_{k=0}^{n-1}\omega_{k}(\overline{\mathcal{F}(+)}),
\]
as required.

Next, by Theorem~4.2(c.ii), $|\overline{\mathcal{B}_n(\times)}|$ is the number of products which have one of the following forms:
\[
   [\varepsilon f_1]\cdots[f_r] \quad\text{or}\quad[\varepsilon f_1]\cdots[f_r][g].
\]
In the first form, $\varepsilon\in\{-1,1\}$, each $f_j\in\mathcal{MB}_{k_j}(+,-)$, and $(k_1,\ldots,k_r)$ is a non-trivial partition of $n$. In the second form, $\varepsilon\in\{-1,1\}$, each $f_j\in\mathcal{MB}_{k_j}(+,-)$, $(k_1,\ldots,k_r)$ is a partition of $k$ for some $1\leqslant k\leqslant n-1$ and $g\in\mathcal{F}_{n-k}(+,\times)$. Hence, the number of AE's in the first form is $2\omega_n'(\overline{\mathcal{MB}(+,-)})$, and in the second form is $2\omega_k(\overline{\mathcal{MB}(+,-)})|\overline{\mathcal{F}_{n-k}(+,\times)}|$. Therefore, the total number of AE's in $\overline{\mathcal{B}_n(\times)}$ is
\[
|\overline{\mathcal{B}_n(\times)}|=2\omega_n'(\overline{\mathcal{MB}(+,-)})+\sum_{k=1}^{n-1}
2\omega_k(\overline{\mathcal{MB}(+,-)})|\overline{\mathcal{F}_{n-k}(+,\times)}|,
\]
as required.

Finally, by Theorem~4.2(c.iii), $|\overline{\mathcal{C}_n(\times)}|$ is the number of products which have one of the following forms:
\begin{gather*}
  [f_1][f_2]\cdots[f_r], \\
  [f_1][f_2]\cdots[f_r][g],\\
  [f_1][f_2]\cdots[f_r][h].
\end{gather*}
In the first form, each $f_j\in\mathcal{C}_{k_j}(+,-)$ and $(k_1,\ldots,k_r)$ is a non-trivial partition of $n$. In the second form, each $f_j\in\mathcal{C}_{k_j}(+,-)$, $(k_1,\ldots,k_r)$ is a partition of $k$ for some $1\leqslant k\leqslant n-1$ and $g\in\mathcal{F}_{n-k}(+,\times)$. In the third form, each $f_j\in\mathcal{C}_{k_j}(+,-)$, $(k_1,\ldots,k_r)$ is a partition of $k$ for some $1\leqslant k\leqslant n-1$ and $h\in\mathcal{MB}_{n-k}(+,-)$. Hence, the number of AE's in the first form is $\omega_n'(\overline{\mathcal{C}(+,-)})$ and the number of AE's in the second form is
\[
\sum_{k=1}^{n-1}\omega_k(\overline{\mathcal{C}(+,-)})|\overline{\mathcal{F}_{n-k}(+,\times)}|.
\]
Observe that $|\overline{\mathcal{MB}_m(\ast)}|=\frac{1}{2}|\overline{\mathcal{B}_m(\ast)}|$ for every $m\geqslant1$ and for every $\ast\in\{+,-,\times,\div\}$. Thus, the number of AE's in the third form is
\[
\sum_{k=1}^{n-1}\omega_k(\overline{\mathcal{C}(+,-)})|\overline{\mathcal{MB}_{n-k}(+,-,\times)}|
=\sum_{k=1}^{n-1}\omega_k(\overline{\mathcal{C}(+,-)})\textstyle\frac{1}{2}|\overline{\mathcal{B}_{n-k}(+,-,\times)}|.
\]
Therefore, the total number of AE's in $\overline{\mathcal{C}_n(\times)}$ is
\[
|\overline{\mathcal{C}_n(\times)}|=\omega_n'(\overline{\mathcal{C}(+,-)})+\sum_{k=1}^{n-1}
\omega_k(\overline{\mathcal{C}(+,-)})\big(|\overline{\mathcal{F}_{n-k}(+,\times)}|
+\textstyle\frac{1}{2}|\overline{\mathcal{B}_{n-k}(+,-,\times)}|\big),
\]
as required.

\

(d) By Theorem~4.2(d.i), $|\overline{\mathcal{F}_n(\div)}|$ is the number of quotients of the form
\[
[g]/[h],
\]
where $g\in\mathcal{F}_k(+,\times)$ and $h\in\mathcal{F}_{n-k}(+,\times)$ for some $1\leqslant k\leqslant n-1$. Hence,
\[
|\overline{\mathcal{F}_n(\div)}|=\sum_{k=1}^{n-1}|\overline{\mathcal{F}_k(+,\times)}||\overline{\mathcal{F}_{n-k}(+,\times)}|,
\]
as required.

Similarly, by Theorem~4.2(d.ii) $|\overline{\mathcal{B}_n(\div)}|$ is the number of quotients which have one of the following forms:
\[
[\varepsilon g]/[h]\quad\text{or}\quad
[g]/[\varepsilon h].
\]
In the first form, $\varepsilon\in\{1,-1\}$, $g\in\mathcal{MB}_{n-k}(+,-,\times)$, and $h\in\mathcal{F}_k(+,\times)\cup\mathcal{MB}_k(+,-,\times)$ for some $1\leqslant k\leqslant n-1$. In the second form, $\varepsilon\in\{1,-1\}$, $g\in\mathcal{F}_k(+,\times)$, and $h\in\mathcal{MB}_{n-k}(+,-,\times)$ for some $1\leqslant k\leqslant n-1$. Hence,
\[
\begin{split}
  |\overline{\mathcal{B}_n(\div)}|=\sum_{k=1}^{n-1}&2|\overline{\mathcal{MB}_{n-k}(+,-,\times)}|
  \big(|\overline{\mathcal{F}_k(+,\times)}|+|\overline{\mathcal{MB}_k(+,-,\times)}|\big)\\
&\qquad\qquad+\sum_{k=1}^{n-1}2|\overline{\mathcal{F}_k(+,\times)}||\overline{\mathcal{MB}_{n-k}(+,-,\times)}|.\\
\end{split}
\]
Since $|\overline{\mathcal{MB}_k(+,-,\times)}|=\frac{1}{2}|\overline{\mathcal{B}_k(+,-,\times)}|$, we obtain that
\[
|\overline{\mathcal{B}_n(\div)}|=\sum_{k=1}^{n-1}|\overline{\mathcal{B}_{n-k}(+,-,\times)}
|\textstyle(2|\overline{\mathcal{F}_k(+,\times)}|
+\frac{1}{2}|\overline{\mathcal{B}_k(+,-,\times)}|),
\]
as required.

Finally, by Theorem~4.2(d.iii), $|\overline{\mathcal{C}_n(\div)}|$ is the number of quotients of the form
\[
[g]/[h],
\]
where either $g\in\mathcal{C}_{n-k}(+,-,\times)$ and $h\in\mathcal{F}_k(+,\times)\cup\mathcal{MB}_k(+,-,\times)\cup\mathcal{C}_k(+,-,\times)$, or
$g\in\mathcal{F}_k(+,\times)\cup\mathcal{MB}_k(+,-,\times)$ and $h\in\mathcal{C}_{n-k}(+,-,\times)$ for some $1\leqslant k\leqslant n-1$. Hence,
\[
\begin{split}
  |\overline{\mathcal{C}_n(\div)}|
  &=\sum_{k=1}^{n-1}|\mathcal{C}_{n-k}(+,-,\times)|\big(|\mathcal{F}_k(+,\times)|+|\mathcal{MB}_k(+,-,\times)|+
|\mathcal{C}_k(+,-,\times)|\big)\\
&\qquad\qquad+\sum_{k=1}^{n-1}\big(|\mathcal{F}_k(+,\times)|+|\mathcal{MB}_k(+,-,\times)|\big)|\mathcal{C}_{n-k}(+,-,\times)|\\
&=\sum_{k=1}^{n-1}|\mathcal{C}_{n-k}(+,-,\times)|\big(2|\mathcal{F}_k(+,\times)|
+|\mathcal{B}_k(+,-,\times)|+|\mathcal{C}_k(+,-,\times)|\big),\\
\end{split}
\]
as required.
\end{proof}

We conclude this section with some tables describing the number of non-isomorphic AE's with $n$ variables for $n\in\{1,2,\ldots,6\}$.
\[
\begin{array}{c||c|c|c|c||c|}
                & \overline{\mathcal{A}_1(+)} & \overline{\mathcal{A}_1(-)} & \overline{\mathcal{A}_1(\times)} & \overline{\mathcal{A}_1(\div)} & \hbox{Total}\\ \hline\hline
  \hbox{Type 1} &      0            &    0              &   1                    & 0 & 1\\ \hline
  \hbox{Type 2} &      0            &    0              &    0                   & 0 & 0\\ \hline
  \hbox{Type 3} &      0            &     0             &     0                  & 0 & 0\\ \hline\hline
  \hbox{Total}  &      0            &     0             &     1                  & 0 & \mathbf{1}\\ \hline
\end{array}
\]
\[
\begin{array}{c||c|c|c|c||c|}
                & \overline{\mathcal{A}_2(+)} & \overline{\mathcal{A}_2(-)} & \overline{\mathcal{A}_2(\times)} & \overline{\mathcal{A}_2(\div)} & \hbox{Total}\\ \hline\hline
  \hbox{Type 1} & 1 & 0 & 1 & 1 & 3\\ \hline
  \hbox{Type 2} & 0 &    0              &     0                   & 0 & 0\\ \hline
  \hbox{Type 3} & 0            &     1             &     0                  & 0 & 1\\ \hline\hline
  \hbox{Total}  & 1            &     1             &     1                  & 1 & \mathbf{4}\\ \hline
\end{array}
\]
\[
\begin{array}{c||c|c|c|c||c|}
                & \overline{\mathcal{A}_3(+)} & \overline{\mathcal{A}_3(-)} & \overline{\mathcal{A}_3(\times)} & \overline{\mathcal{A}_3(\div)} & \hbox{Total}\\ \hline\hline
  \hbox{Type 1} &  3 & 0 & 2 & 4 & 9\\ \hline
  \hbox{Type 2} &  0 & 6 & 0 & 0 & 6\\ \hline
  \hbox{Type 3} &  0 & 0 & 1 & 2 & 3\\ \hline\hline
  \hbox{Total}  &  3 & 6 & 3 & 6 & \mathbf{18}\\ \hline
\end{array}
\]
\[
\begin{array}{c||c|c|c|c||c|}
                & \overline{\mathcal{A}_4(+)} & \overline{\mathcal{A}_4(-)} & \overline{\mathcal{A}_4(\times)} & \overline{\mathcal{A}_4(\div)} & \hbox{Total}\\ \hline\hline
  \hbox{Type 1} &  12 & 0  & 6 & 14 & 32\\ \hline
  \hbox{Type 2} &  3  & 27 & 6 & 12 & 48\\ \hline
  \hbox{Type 3} &  0  & 3  & 3 & 7  & 13\\ \hline\hline
  \hbox{Total}  &  15  & 30  & 15 & 33 & \mathbf{93}\\ \hline
\end{array}
\]
\[
\begin{array}{c||c|c|c|c||c|}
                & \overline{\mathcal{A}_5(+)} & \overline{\mathcal{A}_5(-)} & \overline{\mathcal{A}_5(\times)} & \overline{\mathcal{A}_5(\div}) & \hbox{Total}\\ \hline\hline
  \hbox{Type 1} & 44 & 0   & 21 & 56  & 121 \\ \hline
  \hbox{Type 2} & 37 & 155 & 42 & 96  & 330  \\ \hline
  \hbox{Type 3} & 0  & 3   & 12 & 34  & 49 \\ \hline\hline
  \hbox{Total}  & 81 & 158 & 75 & 186 & \mathbf{500} \\ \hline
\end{array}
\]
\[
\begin{array}{c||c|c|c|c||c|}
                & \overline{\mathcal{A}_6(+)} & \overline{\mathcal{A}_6(-)} & \overline{\mathcal{A}_6(\times)} & \overline{\mathcal{A}_6(\div)} & \hbox{Total}\\ \hline\hline
\hbox{Type 1} & 186 & 0 & 84 & 227 & 497 \\ \hline
\hbox{Type 2} & 295 & 837 & 294 & 690 & 2116 \\ \hline
\hbox{Type 3} &6 & 19 & 51 & 155 & 231 \\ \hline
\hbox{Total}  & 487 & 856 & 429 & 1072 & \mathbf{2844} \\ \hline
\end{array}
\]
As an illustrative example, we shall demonstrate the computation of $|\overline{\mathcal{F}_6(+)}|$, $|\overline{\mathcal{C}_6(-)}|$, and $|\overline{\mathcal{B}_6(\times)}|$ using the formulas of Theorem~4.3. First, the partitions of $n=6$ are:
\begin{gather*}
  (1^6)\quad(1^4,2^1)\quad(1^2,2^2)\quad(2^3)\quad(1^3,3^1) \\
  (1^1,2^1,3^1)\quad(1^2,4^1)\quad(3^2)\quad(2^1,4^1)\quad(1^1,5^1)\quad(6^1)
\end{gather*}
By Theorem~4.3(a) we obtain that
\[
\begin{split}
  |\overline{\mathcal{F}_6(+)}|&
  =\omega_6'(\overline{\mathcal{F}(\times,\div)})
  =\textstyle\multiset{|\overline{\mathcal{F}_1(\times,\div)}|}{6}+
  \multiset{|\overline{\mathcal{F}_1(\times,\div)}|}{4}\multiset{|\overline{\mathcal{F}_2(\times,\div)}|}{1}\\
  &\textstyle+\multiset{|\overline{\mathcal{F}_1(\times,\div)}|}{2}\multiset{|\overline{\mathcal{F}_2(\times,\div)}|}{2}
  +\multiset{|\overline{\mathcal{F}_2(\times,\div)}|}{3}
  +\multiset{|\overline{\mathcal{F}_1(\times,\div)}|}{3}\multiset{|\overline{\mathcal{F}_3(\times,\div)}|}{1}\\
  &\textstyle+\multiset{|\overline{\mathcal{F}_1(\times,\div)}|}{1}\multiset{|\overline{\mathcal{F}_2(\times,\div)}|}{1}
  \multiset{|\overline{\mathcal{F}_3(\times,\div)}|}{1}+\multiset{|\overline{\mathcal{F}_1(\times,\div)}|}{2}\multiset{|\overline{\mathcal{F}_4(\times,\div)}|}{1}\\
  &\textstyle+\multiset{|\overline{\mathcal{F}_3(\times,\div)}|}{2}
  +\multiset{|\overline{\mathcal{F}_2(\times,\div)}|}{1}\multiset{|\overline{\mathcal{F}_4(\times,\div)}|}{1}
  +\multiset{|\overline{\mathcal{F}_1(\times,\div)}|}{1}\multiset{|\overline{\mathcal{F}_5(\times,\div)}|}{1}.
\end{split}
\]
Note that $|\overline{\mathcal{F}_k(\times,\div)}|=|\overline{\mathcal{F}_k(\times)}|+|\overline{\mathcal{F}_k(\div)}|$ for every $k$.
Based on the results in the previous tables we obtain that
\[
\begin{split}
  |\overline{\mathcal{F}_6(+)}|&
  =\overbrace{\dispmultiset{1}{6}}^{1}+
  \overbrace{\dispmultiset{1}{4}\dispmultiset{2}{1}}^{1\cdot2}+\overbrace{\dispmultiset{1}{2}\dispmultiset{2}{2}}^{1\cdot3}
  +\overbrace{\dispmultiset{2}{3}}^{4}
  +\overbrace{\dispmultiset{1}{3}\dispmultiset{6}{1}}^{1\cdot6}\\
  &\quad+\overbrace{\dispmultiset{1}{1}\dispmultiset{2}{1}
  \dispmultiset{6}{1}}^{1\cdot2\cdot6}+\overbrace{\dispmultiset{1}{2}\dispmultiset{20}{1}}^{1\cdot20}
  +\overbrace{\dispmultiset{6}{2}}^{21}+\overbrace{\dispmultiset{2}{1}\dispmultiset{20}{1}}^{2\cdot20}
  +\overbrace{\dispmultiset{1}{1}\dispmultiset{77}{1}}^{1\cdot77}\\
  &=186,
\end{split}
\]
as required. Next, by Theorem~4.3(b) and based on the results in the previous tables, we obtain that
\[
\begin{split}
  |\overline{\mathcal{C}_6(-)}|&=\overbrace{|\overline{\mathcal{C}_4(+,\times,\div)}|}^{0+3+7}
  \overbrace{|\overline{\mathcal{F}_1(+,\times,\div)}|}^{0+1+0}
+\overbrace{|\overline{\mathcal{C}_2(+,\times,\div)}|}^{0+0+0}
\overbrace{|\overline{\mathcal{F}_2(+,\times,\div)}|}^{1+1+1}\\
&\qquad+\underbrace{|\overline{\mathcal{C}_0(+,\times,\div)}|}_{1}
\underbrace{|\overline{\mathcal{F}_3(+,\times,\div)}|}_{3+2+4}=19,
\end{split}
\]
as required. For our last computation, Theorem~4.3(c) implies that
\[
\begin{split}
  |\overline{\mathcal{B}_6(\times)}|
  &=2\omega_6'(\overline{\mathcal{MB}(+,-)})
  +2\omega_1(\overline{\mathcal{MB}(+,-)})|\overline{\mathcal{F}_5(+,\times)}|\\
  &\qquad\quad+2\omega_2(\overline{\mathcal{MB}(+,-)})|\overline{\mathcal{F}_4(+,\times)}|+\cdots
  +2\omega_5(\overline{\mathcal{MB}(+,-)})|\overline{\mathcal{F}_1(+,\times)}|.
\end{split}
\]
Note that similarly to the computation of $\omega_6'(\overline{\mathcal{F}(\times,\div)})$, we obtain that
\[
  \omega_6'(\overline{\mathcal{MB}(+,-)})
  =\textstyle\multiset{|\overline{\mathcal{MB}_1(+,-)}|}{6}
  +\cdots+\multiset{|\overline{\mathcal{MB}_1(+,-)}|}{1}\multiset{|\overline{\mathcal{MB}_5(+,-)}|}{1}.
\]
Since $|\overline{\mathcal{B}_1(+,-)}|=|\overline{\mathcal{B}_2(+,-)}|=0$, and since $\multiset{0}{k}=0$ for every $k>0$, the above sum reduces to
\[
\omega_6'(\overline{\mathcal{MB}(+,-)})=\dispmultiset{|\overline{\mathcal{MB}_3(+,-)}|}{2}
=\dispmultiset{\frac{1}{2}|\overline{\mathcal{B}_3(+,-)}|}{2}
=\dispmultiset{3}{2}=6.
\]
Similarly, the sum $\omega_k(\overline{\mathcal{MB}(+,-)})$ consists of summands that correspond to partitions of $k$ that use neither $1$ nor $2$. Since every non-trivial partition of $1\leqslant k\leqslant 5$ contains either $1$ or $2$, it follows that
\[
2\omega_k(\overline{\mathcal{MB}(+,-)})=2\dispmultiset{|\overline{\mathcal{MB}_k(+,-)}|}{1}=
2\dispmultiset{\frac{1}{2}|\overline{\mathcal{B}_k(+,-)}|}{1}=|\overline{\mathcal{B}_k(+,-)}|
\]
for each $1\leqslant k\leqslant 5$. Thus, $2\omega_1(\overline{\mathcal{MB}(+,-)})=0$, $2\omega_2(\overline{\mathcal{MB}(+,-)})=0$,
$2\omega_3(\overline{\mathcal{MB}(+,-)})=6$, $2\omega_4(\overline{\mathcal{MB}(+,-)})=30$, and $2\omega_5(\overline{\mathcal{MB}(+,-)})=192$. Therefore,
\[
  |\overline{\mathcal{B}_6(\times)}|=2\overbrace{\omega_6'(\overline{\mathcal{MB}(+,-)})}^{6}
  +6\overbrace{|\overline{\mathcal{F}_3(+,\times)}|}^{5}
  +30\overbrace{|\overline{\mathcal{F}_2(+,\times)}|}^{2}+192\overbrace{|\overline{\mathcal{F}_1(+,\times)}|}^{1}=294,
\]
as required.

To conclude this section, in the following table, we compare the values of $|\mathcal{A}_n|$ and $|\overline{\mathcal{A}_n}|$ for $1\leqslant n\leqslant 17$.
\[
\begin{array}{||l||l||l||}
  n & |\mathcal{A}_n| & |\overline{\mathcal{A}_n}|\\ \hline\hline
 1 & 1 & 1 \\
 2 & 6 & 4 \\
 3 & 68 & 18 \\
 4 & 1170 & 93 \\
 5 & 27142 & 500 \\
 6 & 793002 & 2844 \\
 7 & 27914126 & 16621 \\
 8 & 1150212810 & 99674 \\
 9 & 54326011414 & 608448 \\
10 & 2894532443154 & 3770744 \\
11 & 171800282010062 & 23653630\\
12 & 11243812043430330 & 149925328\\
13 & 804596872359480358 & 958737739\\
14 & 62506696942427106498 & 6178529510 \\
15 & 5239819196582605428254 & 40089044100\\
16 & 471480120474696200252970 & 261693178976\\
17 & 45328694990444455796547766 & 1717542967251\\
\end{array}
\]
The sequence $|\mathcal{A}_n|$ corresponds to the values in~\cite{Sloane_A140606}, exhibiting rapid superexponential growth. According to the asymptotic analysis of its exponential generating function, the growth rate is characterized by:
\[
|\mathcal{A}_n|\sim \left(\frac{n}{eb}\right)^n\frac{c\sqrt{b}}{n}
\]
with $b\approx 0.16142$ and $c\approx 1.87722$. The behavior of the quotient $R_n=|\mathcal{A}_n|/|\overline{\mathcal{A}_n}|$ is of particular interest. Because $|\overline{\mathcal{A}_n}|$ counts expressions where operand permutations are disregarded, it is natural to observe $R_n<n!$. Empirical data shows that this ratio grows significantly slower than the full factorial; specifically, the multiplier between successive terms, $M_n=R_n/R_{n-1}$, grows linearly with $n$. For example, $M_{16}\approx 13.78$ and $M_{17}\approx 14.65$. This linear progression, which follows the approximation $M_n\approx 0.86n$, suggests that $|\overline{\mathcal{A}_n}|$ constitutes a larger density than a simple $1/n!$ fraction, pointing toward an asymptotic relationship of the form $R_n\sim n!\cdot c^{-n}$ for a constant $c\approx 1.16$. Characterizing this relationship formally remains an open problem.
\section{The solution of the $21$-puzzle}
Among the $93$ non-isomorphic AE's in $\mathcal{A}_4$, the expression
\[
f=\frac{x_1}{x_2-\frac{x_3}{x_4}}
\]
is the unique AE (up to isomorphism) needed for the solution to the $21$-puzzle. Indeed,
\[
f(6,1,5,7)=\frac{6}{1-\frac{5}{7}}=21,
\]
as required.

\end{document}